\documentclass[11pt,reqno]{amsart}
\usepackage[colorlinks=true,
pdfstartview=FitV, linkcolor=cyan, citecolor=magenta,
urlcolor=blue]{hyperref}
\usepackage{amsmath,amsfonts,latexsym,amssymb}
\usepackage{mathrsfs}
\usepackage[latin1]{inputenc}
\usepackage[T1]{fontenc}
\usepackage{ae,aecompl}
\usepackage{braket}
\usepackage{color}
\usepackage{graphicx}
\usepackage{subfig}
\newtheorem{theorem}{Theorem}[section]
\newtheorem{lemma}[theorem]{Lemma}
\newtheorem{proposition}[theorem]{Proposition}
\newtheorem{corollary}[theorem]{Corollary}

\theoremstyle{definition}

\renewcommand{\leq}{\leqslant}
\renewcommand{\geq}{\geqslant}

\newcommand{\sln}{\mathsf{SL}_d(\mathbb R)}
\newcommand{\psln}{\mathsf{PSL_d(\mathbb R)}}

\newcommand{\grf}{\pi_1(S)}
\newcommand{\bgrf}{\partial_\infty\pi_1(S)}

\newcommand{\tr}{\operatorname{Tr}}

\newcommand{\TT}{\mathsf{T}}
\newcommand{\II}{{\bf I}}
\newcommand{\JJ}{{\bf J}}
\newcommand{\PP}{{\bf P}}
\newcommand{\p}{{\bf{p}}}

\newcommand{\C}{\mathbb C}

\newcommand{\Real}{\mathbb R}

\DeclareMathOperator{\tp}{t}
\DeclareMathOperator{\pb}{B_d(S)}

\newcommand{\eps}{\varepsilon}
\newcommand{\ms}{\mathsf}

\renewcommand{\hom}{\operatorname{Hom}}
\newcommand{\g}{\gamma}
\renewcommand{\H}{\mathbb H}

\renewcommand{\P}{\mathbb P}

\title[Pressure metrics]{An introduction to pressure metrics for higher Teichm\"uller spaces}
\author[Bridgeman]{Martin Bridgeman}
\address{Boston College, Chestnut Hill, MA 02467 USA}
\author[Canary]{Richard Canary}
\address{University of Michigan, Ann Arbor, MI 41809 USA}
\author[Sambarino]{Andr\'es Sambarino}
\address{Universit\'e Pierre et Marie Curie (Paris VI), 75005 Paris France}
\thanks{Bridgeman  was partially supported by NSF grant DMS - 1500545, Canary was partially supported by NSF grant DMS - 1306992 and  Sambarino was partially supported by
the European Research Council under the {\em European Community}'s
 seventh Framework Programme (FP7/2007-2013)/ERC {\em grant agreement} ${\rm n^o}$ FP7-246918.The authors  also acknowledge  support from U.S. National Science Foundation grants 
DMS 1107452, 1107263, 1107367 "RNMS: GEometric structures And Representation varieties" (the GEAR Network)}

\begin{document}
\maketitle

\tableofcontents

\section{Introduction}

We discuss how one uses the Thermodynamic Formalism to produce metrics on higher Teichm\"uller spaces. Our higher Teichm\"uller spaces will be 
spaces of Anosov representations of a word hyperbolic group into a semi-simple Lie group. To each such representation we associate an 
Anosov flow encoding eigenvalue information, and the Thermodynamic Formalism gives us a way to measure the difference between two such flows. 
This difference gives rise to an analytic semi-norm, which in many cases turns out to be a Riemannian metric, called the \emph{pressure metric}.  
This paper surveys results of  Bridgeman-Canary-Labourie-Sambarino \cite{BCLS} and discusses questions and open problems which arise. 

We begin by discussing our construction in the classical setting of the Teichm\"uller space of a closed orientable surface of genus at least 2. 
In this setting, our construction agrees with 
Thurston's Riemannian metric, as re-interpreted by Bonahon \cite{bonahon} using geodesic currents  and  McMullen \cite{mcmullen-pressure} using the Thermodynamic Formalism.
Wolpert \cite{wolpert-thurston}  showed that Thurston's metric is a multiple of the Weil-Petersson metric. The key difference between our
approach and McMullen's is that we work directly with the geodesic flow of the surface, rather than with Bowen-Series coding of the action
of the group on the limit set. Since such a coding is not known to exist for every hyperbolic group,
this approach will be crucial to generalizing  our results to the setting of all hyperbolic groups.

We next discuss the construction of the pressure metric in the simplest new situation: the Hitchin component of representations of a surface group into
$\psln$. This setting offers the cleanest results and also several simplifications of the general proof. 
Given a Hitchin representation, inspired by earlier work of Sambarino \cite{sambarino-quantitative},
we construct a metric Anosov flow, which we call the geodesic flow of the representation,
whose periods record the spectral radii of the elements in the image. 
We obtain a mapping class group invariant Riemannian metric on 
a Hitchin component whose restriction to the Fuchsian locus is a multiple of the Weil-Petersson metric. 

We hope that the discussion of the pressure metric in these two simpler settings will provide motivation and intuition for the general construction.
In section \ref{general} we discuss the more general settings studied in \cite{BCLS} with some comments on the additional difficulties which must be overcome.
We finish with a discussion  of open problems.

{\bf Acknowledgements:} The paper is based on a Master Class given by the authors at the Centre for Quantum Geometry of
Moduli Spaces in Aarhus.
We thank Jorgen Anderson for the invitation to give this Master Class and the editors for the invitation to write this
article. We thank Marc Burger, Fran\c cois Labourie and Adam Sikora for helpful conversations. Substantial portions of
this paper were written while the authors were in residence at the Mathematical Sciences Research Institute in Berkeley, CA,
during the Spring 2015 semester and were partially supported by NSF grant No. 0932078 000.

\section{The Thermodynamic Formalism}

The Thermodynamic Formalism was introduced by Bowen and Ruelle (\cite{bowen-book, BowenRuelle, ruelle}) as a tool to study the ergodic theory  of Anosov flows and diffeomorphisms. It was further developed by Parry and Pollicott, their monograph \cite{parry-pollicott} is a standard reference for the material covered here. McMullen \cite{mcmullen-pressure} introduced the pressure form as a tool for constructing metrics on spaces which may be mapped into H\"older potentials over a shift-space. We will give a quick summary of the basic facts we will need, but we encourage the reader to consult the original references and
the more complete discussion and references in \cite{BCLS}.

We recall that a smooth flow $\phi=(\phi_t:X\to X)_{t\in\mathbb R}$ on a compact Riemannian manifold is said to be 
{\em Anosov} if there is a flow-invariant splitting $TX = E^s\oplus E_0\oplus E^u$ where $E_0$ is a line
bundle parallel to the flow and if $t>0$, then $d\phi_t$ is exponentially contracting on $E_+$ and $d\phi_{-t}$ is exponentially contracting on $E_-.$ We will always assume that our Anosov flows are {\em topologically transitive} (i.e. have a dense orbit).
It is a celebrated theorem of Anosov (see \cite[Thm. 17.5.1]{katok-hasselblatt}) that the geodesic flow of a closed hyperbolic surface, and more generally of a
closed negatively curved manifold, is a  topologically transitive Anosov flow.

\subsection{Entropy, pressure and orbit-equivalence} Let $\phi$ be  a topologically transitive Anosov flow on a compact
Riemannian manifold $X$. If $a$ is a $\phi$-periodic orbit, denote by $\ell(a)$ its period and let $$R_T=\{a\textrm{ closed orbit}\mid \ell(a)\leq T\}.$$ Then, following Bowen \cite{bowen1}, we may define the {\em topological entropy} of $\phi$ to be the exponential growth rate of the number of periodic orbits whose periods are at most $T$, i.e. $$h(\phi)= \limsup_{T \rightarrow \infty} \frac{\log \#R_T}{T}.$$ 
Moreover, if $g:X\to\Real$ is H\"older and $a$ is a closed orbit, denote by $$\ell_g(a)=\int_0^{\ell(a)} g(\phi_s(x))\ ds,$$ where $x$ is any point on $a.$ Then, following Bowen-Ruelle \cite{BowenRuelle}, we may define the \emph{topological pressure} of $g$ (or simply \emph{pressure}) by $$\PP(g)=\PP(\phi,g)= \limsup_{T\to\infty}\frac1T\log\left(\sum_{a\in R_T}e^{\ell_g(a)}\right).$$ 

Note that $\PP(g)$ only depends on the periods of $g,$ i.e. the collection of numbers $\{\ell_g(a)\}.$ 

Liv\v sic provides a \emph{pointwise} relation for two functions having the same periods: two H\"older functions $f,g:X\to\mathbb R$ are {\em Liv\v sic cohomologous} if there exists a H\"older function $V:X\to \mathbb R,$ which is $C^1$ in the direction of the flow $\phi,$ such that \begin{equation}\label{cohomologous}f(x)-g(x)=\left.\frac{\partial}{\partial t}\right|_{t=0}V(\phi_t(x)).\end{equation}
Liv\v sic \cite{livsic} proved the following fundamental result:

\begin{theorem}[Liv\v sic \cite{livsic}] 
\label{livsic}
If $\phi$ is a topologically transitive Anosov flow and $g:X\to\Real$ is a H\"older function
such that $\ell_g(a)=0$ for every closed orbit a, then $g$ is Liv\v sic cohomologous to 0.
\end{theorem}

Given a positive H\"older function $f:X\to (0,\infty)$ one may define a {\em reparametrization} of the
flow so that its ``speed'' at a point $x$ is multiplied by $f(x)$.
More formally, let $$\kappa_f(x,t) = \int_0^t f(\phi_s(x)) ds,$$ and define $\phi^f=(\phi_t^f:X\to X)_{t\in\mathbb R}$ so that 
$\phi^f_{\kappa_f(x,t)}(x) = \phi_t(x).$ In particular, if $a$ is a $\phi$-closed orbit then $a$ is also a closed orbit of the flow $\phi^f$ with period $\ell_f(a).$

A \emph{H\"older orbit equivalence} between two flows is a H\"older homeomorphism that sends orbits to orbits.
Moreover, if it preserves time, it is called a \emph{H\"older conjugacy}. In particular, the identity map is a H\"older
orbit equivalence from $\phi$ to $\phi^f$ when $f$ is a positive H\"older function.
(In our setting, the inverse of any
H\"older orbit equivalence is also a H\"older orbit equivalence, but we will not need this fact.)

Liv\v sic's theorem implies that two positive H\"older functions are Liv\v sic cohomologous if and only if the periods 
of $\phi^f$ and $\phi^g$ agree. If this is the case, the function $V$ in Equation (\ref{cohomologous})  provides a 
H\"older conjugacy between $\phi^f$ and $\phi^g.$ Moreover, one has the following standard consequence of 
Liv\v sic's Theorem (see Sambarino \cite[Lemma 2.11]{sambarino-orbital}).

\begin{lemma}\label{time-conjugacy} 
If $\phi$ is a topologically transitive, Anosov flow and there exists H\"older orbit equivalence to
a H\"older-continuous flow $\psi$,  
then there exists a H\"older function $f:X\to(0,\infty)$ such that $\psi$ is H\"older conjugate to $\phi^f.$
\end{lemma}

The flow $\phi^f$ remains topologically transitive and  is again Anosov but in a metric sense. 
More specifically, $\phi^f$ is a Smale flow in
the sense of Pollicott \cite{pollicott-smale}. Pollicott shows that all the results we rely on in
the ensuing discussion generalize to the setting of Smale flows, which we will refer to as metric Anosov
flows.
Hence,  we may define the topological entropy of $\phi^f$ as
$$h(f)  = \limsup_{T \rightarrow \infty} \frac{\log \#R_T(f)}{T},$$ 
where $R_T(f)=\{a\textrm{ closed orbit}\mid \ell_f(a)\leq T\}.$ 
We recall the following standard lemma which relates pressure and entropy.

\begin{lemma}
\label{entropia2}
{\rm (see Sambarino \cite[Lemma 2.4]{sambarino-quantitative})}  
If $\phi$ is a topologically transitive Anosov flow and $f:X\to(0,\infty)$
is H\"older, then $\PP(-h f)=0$ if and only if $h=h(f)$. 
\end{lemma}

Ruelle \cite[Cor. 7.10]{ruelle} (see also Parry-Pollicott \cite[Prop. 4.7]{parry-pollicott})
proved that the pressure function $\PP(g)$ is a real analytic function of the H\"older function $g$. 
It follows from Lemma \ref{entropia2}
and the Implicit Function Theorem that entropy varies analytically in $f$. 
Ruelle \cite{ruelle-hd} used a similar observation to show that the Hausdorff dimension of a quasifuchsian Kleinian group varies analytically on quasifuchsian space.

Ruelle \cite{ruelle} also showed that $\PP$ is a convex function and thus if $f:X\to\Real$ and $g:X\to\Real$  are H\"older functions,
 $$\left.\frac{\partial^2}{\partial t^2}\right|_{t=0} \PP(f+tg)\geq0.$$ 
Consider the space 
$$\mathcal P(X)=\{\Phi:X\to\Real\textrm{ H\"older}\mid\PP(\Phi)=0\}$$ of pressure zero H\"older functions on $X$.
It follows immediately from the definition above that 
the pressure function $\PP$ is constant on Liv\v sic cohomology classes, so it is  natural to consider the space 
$$\mathcal{H}(X)=\mathcal P(X)/\sim$$ 
of Liv\v sic cohomology classes of pressure zero functions.

McMullen \cite{mcmullen-pressure} defined a pressure semi-norm on the tangent space of the space of pressure zero
H\"older functions on a shift space.
Similarly, 
we define a {\em pressure semi-norm} on $\TT_f\mathcal{P}(X)$,  by letting
$$\|g\|_\PP^2=\left(\left.\frac{\partial^2}{\partial t^2}\right|_{t=0} \PP(f+tg)\right)
\left(\frac{-1}{\left.\frac{\partial}{\partial t}\right|_{t=0} \PP(f+tf)}\right)
$$
for all \hbox{$g\in \TT_f\mathcal P(X)=\ker D_f\PP$}.
(Formally, one should consider the space $\mathcal P^\alpha(X)$ of $\alpha$-H\"older pressure zero functions for some $\alpha>0$.
In all our applications, we will consider embeddings of analytic manifolds into $\mathcal P(X)$ such that  every point has a neighborhood which maps into
$\mathcal P^\alpha(X)$ for some $\alpha>0$. We will consistently suppress this technical detail.)
 
One obtains the following characterization of degenerate vectors, due to Ruelle and Parry-Pollicott.

\begin{theorem}{\rm (Ruelle \cite{ruelle}, see also Parry-Pollicott \cite[Prop. 4.12]{parry-pollicott})}
Let $\phi$ be a topologically transitive Anosov flow and consider $g\in \TT\mathcal P(X).$ Then, $\|g\|_{\PP}=0$ if and only if 
$g$ is Liv\v sic cohomologous to zero, i.e. $\ell_g(a)=0$ for every closed orbit $a.$
\end{theorem}

We make use of the following nearly immediate corollary of this characterization (see the proof of \cite[Lemma 9.3]{BCLS}).

\begin{corollary}
\label{degenerate vectors}
Let $\phi$ be a topologically transitive  Anosov flow.  Suppose
that $\{f_t\}_{t\in (-1,1)}:X\to(0,\infty)$ is
a smooth one parameter family of H\"older functions. Consider $\Phi:(-1,1)\to \mathcal P(X)$ defined by
 $\Phi(t)=-h(f_t)f_t$. Then $\|\dot{\Phi}_0\|_\PP=0$ if and only if $$\left.\frac{\partial}{\partial t}\right|_{t=0} h(f_t)\ell_{f_t}(a)=0$$ for every closed orbit $a$ of $\phi.$
\end{corollary}

\subsection{Intersection and Pressure form}

Inspired by Bonahon's \cite{bonahon} intersection number, 
we define the {\em intersection number} $\II$ of two positive H\"older functions $f_1,f_2:X\to(0,\infty)$ by
 $$\II(f_1,f_2)=\lim_{T\to\infty} \frac{1}{\#R_T(f_1)}\sum_{a\in R_T(f_1)} \frac{\ell_{f_2}(a)}{\ell_{f_1}(a)}$$ 
and their {\em renormalized intersection number} by $$\JJ(f_1,f_2)=\frac{h(f_2)}{h(f_1)}\ \II(f_1,f_2).$$

Bowen's equidistribution theorem on periodic orbits \cite{bowen1} implies that $\II$, and hence $\JJ$, are well-defined
(see \cite[Section 3.4]{BCLS}). One may  use the analyticity of the pressure function and 
results of Parry-Pollicott \cite{parry-pollicott} and
Ruelle \cite{ruelle} to check that they are analytic functions.

\begin{proposition}
\label{entropy analytic}
{\rm (\cite[Prop. 3.12]{BCLS})} 
Let $\phi$ be a topologically transitive Anosov flow and let $\{f_u\}_{u\in M}$ and $\{g_u\}_{u\in M}$ be two analytic familes of positive H\"older functions on $X.$ Then $h(f_u)$ varies analytically over $M$ and $\II(f_u,g_u)$ and $\JJ(f_u,g_u)$
vary analytically over $M\times M$.
\end{proposition}

The seminal work of Bowen and Ruelle \cite{BowenRuelle} may be used
to derive
the following crucial rigidity property for the renormalized intersection number.

\begin{proposition}
\label{renormalized rigidity}
{\rm (\cite[Prop. 3.8]{BCLS})} 
If $\phi$ is a topologically transitive Anosov flow  on $X$ and $f$ and $g$ are positive H\"older functions on $X$,
then 
$$\JJ(f,g)\ge 1.$$
Moreover, $\JJ(f,g)=1$ if and only if $h(f)f$ and $h(g)g$ are Liv\v sic cohomologous.
\end{proposition}

If $\{f_u\}_{u\in M}$ is an analytic family of positive H\"older functions on $X$, then, for all $u\in M$,
we consider the function
$$\JJ_u:M\to \mathbb R$$ 
given by $\JJ_u(v)=\JJ(u,v)$ for all $v\in M$. Proposition \ref{renormalized rigidity} implies that the
Hessian of $J_u$ gives a non-negative bilinear form on $T_uM$. Lemma \ref{entropia2} allows us to
define a thermodynamic mapping 
$$\Phi:M\to \mathcal P(X)$$
by letting $\Phi(u)=-h(f_u)f_u$
The following important, but fairly simple, result shows that this bilinear form is the pull-back
of the pressure form.

\begin{proposition}
{\rm (\cite[Prop. 3.11]{BCLS})}
\label{pressure form is a Hessian}
Let $\phi$ be a topologically transitive  Anosov flow. If $\{f_t\}_{t\in (-1,1)}$ is
a smooth one parameter family of positive H\"older functions on $X$ and $\Phi:(-1,1)\to \mathcal P(X)$ is
given by $\Phi(t)=-h(f_t)f_t$, then 
$$\|\dot{\Phi}_0\|_\PP^2=\left.\frac{\partial^2}{\partial t^2}\right|_{t=0}\JJ(f_0,f_t).$$
\end{proposition}

\section{Basic strategy}

Our basic strategy is inspired by McMullen's \cite{mcmullen-pressure} re-interpretation of Thurston's Riemannian metric
on Teichm\"uller space and its generalization to quasifuchsian space by Bridgeman \cite{bridgeman-wp}.

We consider a family  $\{\rho_u:\Gamma\to \ms G\}_{u\in M}$ of (conjugacy classes of )
representations of a word hyperbolic group $\Gamma$
into  a semi-simple Lie group $\ms G$ parametrized by an analytic manifold $M$. We recall that Gromov \cite{gromov}
associated a geodesic flow $\phi=\{\phi_t:U_\Gamma\to U_\Gamma\}_{t\in\mathbb R}$ to a hyperbolic group $\Gamma$
which agrees with the geodesic flow on $\TT^1X$ in the case when $\Gamma$ is the fundamental group of
a negatively curved manifold $X$ (see Section \ref{general} for details).

In our two basic examples $\Gamma=\pi_1(S)$ where $S$ is a closed, oriented surface of genus at least $2$,
and the flow $\phi$ is the geodesic flow on a hyperbolic surface homeomorphic to $S$.  
The first example will be the classical Teichm\"uller space $\mathcal T(S)$ of hyperbolic structures on $S$, where 
\hbox{$\ms G=\ms{PSL}_2(\mathbb R)$} and \hbox{$M=\mathcal T(S)$}. 
The  second is  the Hitchin component $\mathcal H_d(S)$,
where \hbox{$\ms G=\psln$} and \hbox{$M=\mathcal H_d(S)$}.

\medskip\noindent
{\bf Step 1:}  {\em Associate to each representation $\rho_u$ a  topologically transitive 
metric Anosov flow $\phi^{\rho_u}$ which is H\"older orbit equivalent to the geodesic flow  $\phi$ of $\Gamma$
so that the period of the orbit associated to $\gamma\in\Gamma$ is the ``length'' of $\rho(\gamma)$.}

\medskip

In the case of $\mathcal T(S)$, $\phi^\rho$ will be the geodesic flow of the surface $X_\rho=\mathbb H^2/\rho(S)$.
In the case of a Hitchin component, we will construct a geodesic flow and our notion of length will be the logarithm
of the spectral radius.

If $u\in M$, Lemma \ref{time-conjugacy} provides a positive H\"older function $f_u:U_\Gamma\to \mathbb R$, well-defined up to
Liv\v sic cohomology, such that $\phi^{\rho_u}$ is H\"older
conjugate to $\phi^{f_u}$.

\medskip\noindent
{\bf Step 2:} {\em Define a thermodynamic mapping $\Phi:M\to \mathcal{H}(U_\Gamma)$ by letting
$\Phi(u)=[-h(f_u)f_u]$ and prove that it
has locally analytic lifts, i.e. if \hbox{$u\in M$}, then there exists a neighborhood $U$ of $u$ in $M$ and an
analytic map \hbox{$\tilde\Phi:U\to \mathcal{P}(U_\Gamma)$} which is a lift of $\Phi|_U$.}

\medskip

We may also define a renormalized intersection number on $M\times M$, by letting $\JJ(u,v)=\JJ(f_u,f_v)$.

\medskip\noindent
{\bf Step 3:} {\em Define a pressure form on $M$ by pulling back the pressure from on $\mathcal P(U_\Gamma)$ by (the lifts of) $\Phi$.}

\medskip

Lemma \ref{pressure form is a Hessian} allows us to reinterpret the pull-back of the pressure form
as the Hessian of the renormalized intersection number function.

\medskip\noindent
{\bf Step 4:} {\em Prove, using Corollary \ref{degenerate vectors},
that the resulting pressure form is non-degenerate so gives rise
to an analytic Riemannian metric on $M$.}

\medskip

Step 4 can fail in certain situations. For example, Bridgeman's pressure metric on quasifuchsian
space \cite{bridgeman-wp} is degenerate exactly on the set of pure bending vectors on the Fuchsian locus. However, Bridgeman's
pressure metric still gives rise to a path metric.

\medskip\noindent
{\bf Historical remarks:} Thurston's constructed a Riemannian metric which he describes as the ``Hessian of the length of a random geodesic.''
Wolpert's formulation \cite{wolpert-thurston} of this construction agrees with the  Hessian of the intersection number of the geodesic flows.
From this viewpoint, one regards $\II(\rho,\eta)$, as the length in $X_\eta$ of a random unit length geodesic on $X_\rho$.
If one considers a sequence $\{\gamma_n\}$ of closed geodesics
on $X_\rho$ which are becoming equidistributed (in the sense that $\{{\gamma_n\over \ell_\rho(\gamma_n)}\}$
converges, in the space of geodesic currents on $S$, to the Liouville current $\nu_\rho$ of $X_\rho$),
then 
$$\II(f_\rho,f_\eta)=\lim \frac{\ell_{\eta}(\gamma_n)}{\ell_\rho(\gamma_n)}.$$
Bonahon \cite{bonahon} reinterprets this to say that 
$$\II(f_\rho,f_\eta)=i(\nu_\rho,\nu_\eta)$$
where $i$ is the geometric intersection pairing on the space of geodesic currents.

Bridgeman and Taylor \cite{bridgeman-taylor} used Patterson-Sullivan theory to show that the Hessian of the renormalized intersection
number is a non-negative form on quasifuchsian case.
McMullen  \cite{mcmullen-pressure} then introduced the use of the techniques of Thermodynamic Formalism 
to interpret both of these metrics
as pullbacks of the pressure metric on the space of suspension flows on the shift space associated to the Bowen-Series coding.
Bridgeman \cite{bridgeman-wp} then showed that the resulting pressure form on quasifuchsian space is degenerate exactly on the
set of pure bending vectors on the Fuchsian locus.

\section{The pressure metric for Teichm\"uller space}

In this section, we survey the construction of the pressure metric for the Teichm\"uller space $\mathcal{T}(S)$  of
a closed  oriented surface $S$ of genus $g\ge 2$. 

We recall that $\mathcal{T}(S)$ may be defined as the unique connected component of
$${\rm Hom}(\pi_1(S),\ms{PSL}_2(\mathbb R))/\ms{PGL}_2(\mathbb R)$$
which consists of discrete and faithful representations. If $\rho\in \mathcal{T}(S)$,
then one obtains a hyperbolic surface $X_\rho=\mathbb H^2/\rho(\pi_1(S))$ by regarding
${\rm PSL}_2(\mathbb R)$ as the space of orientation-preserving isometries of the hyperbolic
plane $\mathbb H^2$. 

\subsection{Basic facts}
\label{basic facts}
It is useful to isolate the facts that will make the construction much simpler in this case.
All of these facts will fail even in the setting of the Hitchin component.

\begin{enumerate}
\item The space $\TT^1\H^2$ is canonically identified with the space of ordered triplets on the visual boundary $\partial_\infty\H^2,$ 
$$(\partial_\infty\H^2)^{(3)}=\{(x,y,z)\in (\partial_\infty\H^2)^3:x<y<z\},$$ where $<$ is defined by a given orientation on the topological circle $\partial_\infty\H^2,$
and $(x,y,z)$ is identified with the unit tangent vector to the geodesic $L$ joining $x$ to $z$ at the point which is the orthogonal projection of
$y$ to $L$. 
\item
The surface $X_\rho$ is closed (since it is a surface homotopy equivalent to a closed surface). In fact, by Baer's Theorem, it is diffeomorphic to $S$.
The geodesic flow $\phi^\rho$ on $\TT^1X_\rho$ is thus a topologically transitive Anosov flow on a closed manifold.

\item
The topological entropy $h(\rho)$ of $\phi^\rho$ is equal to 1 (in particular, constant).
\end{enumerate}

Fact (1) is quite straight-forward: if $(p,v)\in\TT^1\H^2,$ denote by $v_\infty\in\partial_\infty\H^2$ the limit at $+\infty$
of the geodesic ray starting at $(p,v),$ then the identification is 
$$(p,v)\mapsto ((-v)_\infty, (iv)_\infty,v_\infty)$$ 
where $iv\in\TT^1\H^2$ is such that the base $\{v,iv\}$ is orthogonal and oriented.

Fact (3) is a standard consequence of the fact, due to Manning \cite{manning}, that the entropy of
the geodesic flow of a negatively curved manifold agrees with the exponential rate of volume growth of
a ball of radius $T$ in its universal cover. In this setting, the universal cover is always $\mathbb H^2$ so the entropy is 
always 1.

\medskip\noindent
{\bf Conventions:} 
For the remainder of the section we fix $\rho_0\in\mathcal{T}(S)$ and identify $S$ with $X_{\rho_0}$. We then obtain an identification of 
$\bgrf$ with $\partial_\infty \H^2$ and of $\TT^1S$ with $\TT^1X_{\rho_0}$. Let $\phi=\phi^{\rho_0}$ be the geodesic flow on $S$.

It will be useful to choose an analytic lift $s:\mathcal{T}(S)\to \rm{Hom}(\pi_1(S),\ms{PSL}_2(\mathbb R)).$ 
In order to do so, we pick non-commuting elements
$\alpha$ and $\beta$ in $\pi_1(S)$ and choose a representative $\rho=s([\rho])$ of $[\rho]$  
such that $\rho(\alpha)$ has attracting fixed point 
$+\infty\in\partial_\infty\mathbb H^2$ and repelling fixed point 0, 
while $\rho(\beta)$ has attracting fixed point 1. From now on, we will implicitly identify
$\mathcal{T}(S)$ with $s(\mathcal{T}(S))$.
This choice will allow us to define our thermodynamic mapping into  the space $\mathcal{P}(\TT^1S)$ 
of pressure zero H\"older functions on $\TT^1S$, rather than just into the space
$\mathcal H(\TT^1S)$ of Liv\v sic cohomology classes of pressure zero H\"older functions on $X$.

\subsection{Analytic variation of limit maps}
It is well known that any two Fuchsian representations are conjugate by a unique H\"older map.

\begin{proposition}
\label{morse} 
If $\rho,\eta\in\mathcal{T}(S)$,
then there is a unique $(\rho,\eta)$-equivariant H\"older homeomorphism
$\xi_{\rho,\eta}:\partial_\infty\H^2\to\partial_\infty\H^2$. Moreover, $\xi_{\rho,\eta}$ varies analytically
in $\eta$.
\end{proposition}

\begin{proof}
By fact (2), there exists a diffeomorphism $h:X_\rho\to X_\eta$ in the homotopy class determined
by $\eta\circ \rho^{-1}$. Choose a $(\rho,\eta)$-equivariant lift $\tilde h:\mathbb H^2\to\mathbb H^2$ of $h$.
Since $\tilde h$ is quasiconformal, classical results in complex analysis 
(see Ahlfors-Beurling \cite{ahlfors-beurling}), imply that $\tilde h$
extends to a quasisymmetric map $\xi_{\rho,\eta}:\partial_\infty\H^2\to\partial_\infty\H^2.$ 
In particular, $\xi_{\rho,\eta}$ is a H\"older homeomorphism.
Since $\tilde h$ is $(\rho,\eta)$-equivariant, so is $\xi$.  
The resulting map is unique, since, by equivariance, if $\gamma\in \pi_1(S)$, then $\xi_{\rho,\eta}$ must take the 
attracting fixed point of $\rho(\gamma)$ to the attracting fixed point of $\eta(\gamma)$.

A more modern approach to the existence of $\xi_{\rho,\eta}$ 
uses the fact that $\mathbb H^2$ is a proper hyperbolic metric space with
boundary $\partial_\infty \mathbb H^2$ and that quasi-isometries of proper hyperbolic metric spaces
extend to H\"older homeomorphisms of their boundary. Since $h$ is a bilipschitz homeomorphism,
it lifts to a bilipschitz homeomorphism of $\mathbb H^2$. In particular, $\tilde h$ is a quasi-isometry of $\mathbb H^2$.

It is a classical result in Teichm\"uller
theory that $\xi_{\rho,\eta}$ varies analytically in $\eta$. 
A more modern, but still complex analytic, approach uses holomomorphic motions and is sketched by 
McMullen \cite[Section 2]{mcmullen-pressure}.  One allows $\eta$ to vary over the space 
$QF(S)$ of (conjugacy classes of) convex cocompact (i.e. quasifuchsian) representations of 
$\pi_1(S)$ into ${\rm PSL}_2(\C)$. 
(Recall that $QF(S)$ is an open neighborhood of $\mathcal T(S)$ in the $\ms{PSL}_2(\C)$-character
variety of $\pi_1(S)$.)
If $\eta\in QF(S)$, there is a $(\rho,\eta)$-equivariant embedding 
$\xi_{\rho,\eta}:\partial_\infty\mathbb H^2\to\widehat{\mathbb C}$ whose
image is the limit set of $\eta(\pi_1(S))$. 
If $z\in \partial_\infty\mathbb H^2$ is a fixed point of a non-trivial element
$\rho(\gamma)$, then $\xi_{\rho,\eta}(z)$ varies holomorphically in $\eta$. Slodkowski's generalized Lambda Lemma
\cite{slodkowski} then implies that $\xi_{\rho,\eta}$ varies complex analytically as $\eta$ varies over $QF(S)$,
and hence varies real analytically  as $\eta$ varies over $\mathcal T(S)$. 

One may also prove analyticity by using techniques of Hirsch-Pugh-Shub \cite{hirsch-pugh-shub} 
as discussed in the next section.
\end{proof}

\subsection{The thermodynamic mapping} 

The next proposition allows us to construct the thermodynamic mapping we use to define the pressure metric.

\begin{proposition}
\label{proposition:function} For every $\eta\in\mathcal T(S),$ there exists a positive H\"older function 
$f_\eta:\TT^1S\to\mathbb (0,\infty)$ such that  
$$\int_{[\g]}f_\eta=\ell_\eta(\g)$$
for all $\g\in\pi_1(S)$. 
Moreover, $f_\eta$ varies analytically in $\eta$.
\end{proposition}

\begin{proof}
Let $\xi_{\rho_0,\eta}$ be the $(\rho_0,\eta)$-equivariant map provided by Proposition \ref{morse}. 
The identification of $\TT^1\mathbb H^2$ with $\bgrf^{(3)}$ gives a $(\rho_0,\eta)$-equivariant H\"older homeomorphism 
$\tilde\sigma:\TT^1\H^2\to\TT^1\H^2$ defined by 
$$\tilde \sigma(x,y,z)=(\xi_{\rho_0,\eta}(x),\xi_{\rho_0,\eta}(y),\xi_{\rho_0,\eta}(z)).$$

Since $\tilde\sigma$ is a $(\rho_0,\eta)$-equivariant map sending geodesics to geodesics, 
the quotient $\sigma:\TT^1S\to\TT^1X_\eta$ is a H\"older orbit equivalence between the geodesic flows $\phi=\phi^{\rho_0}$ and $\phi^{\eta}.$ 
Lemma \ref{time-conjugacy} gives the existence of a function $f_\eta,$ but in order to establish the analytic variation
we give an explicit construction.

If $a,b,c,d \in \partial_\infty \mathbb H^2$, then the signed-distance between the orthogonal projections of $b$ and $c$ 
onto the geodesic with endpoints $a$ and $b$ is $\log|B(a,b,c,d)| $ where 
$$B(a,b,c,d) = \frac{(a-c)(a-d)}{(b-d)(b-c)}$$  
is the cross-ratio. Let 
$$\kappa_{\rho,\eta}((x,y,z),t) = \log(B(\xi_{\rho_0,\eta}(x),\xi_{\rho_0,\eta}(z), \xi_{\rho_0,\eta}(y),\xi_{\rho_0,\eta}(u_t(x,y,z))$$
where  $u_t$ is determined
by $\phi^{\rho_0}_t(x,y,z) = (x, u_t(x,y,z), z)$. We average $\kappa_{\rho,\eta}$ over intervals of length one in the flow to obtain
$$\kappa_{\rho_0,\eta}^1((x,y,z),t)=\int_0^1 \kappa_{\rho_0,\eta}((x,y,z),t+s)\ ds.$$
Then 
$$f_\eta(x,y,z)=\left.\frac{\partial}{\partial t}\right|_{t=0} \kappa_{\rho_0,\eta}^1((x,y,z),t)$$
is H\"older and varies analytically in $\eta$.

One may also prove the analyticity of the reparametrizations in this setting using the
techniques of Katok-Knieper-Pollicott-Weiss \cite{KKPW}.
\end{proof}

Since $h(\phi^\eta)=h(f_\eta)=1,$  by Fact (3) in Section \ref{basic facts}, Lemma \ref{entropia2} implies that 
$\PP(-f_\eta)=0$, where $\PP$ is the pressure function associated to the geodesic flow $\phi^{\rho_0}$ on
our base surface $S=X_{\rho_0}$.
Hence, Proposition \ref{proposition:function}  provides an analytic map $\Phi:\mathcal T(S)\to \mathcal P(\TT^1S)$ 
from the Teichm\"uller space $\mathcal{T}(S)$ to the space $\mathcal P(\TT^1S)$ 
of pressure zero functions on the unit tangent bundle $\TT^1S$, given by
$$\Phi(\eta)=-f_\eta.$$
We call $\Phi$ {\em the thermodynamic mapping.} We note that $\Phi$ depends on our choice of 
$\rho_0 \in \mathcal{T}(S)$ and  on the lift $s:\mathcal{T}(S)\to \ms{Hom}(\pi_1(S),\ms{PSL}_2(\mathbb R)).$

\subsection{The pressure metric}
\label{teich-pressure}
We may then define a pressure form on $\mathcal T(S)$ by pulling back the pressure form
on $\mathcal P(\TT^1S)$. Explicitly, if  $\{\eta_t\}_{t\in (-1,1)}$ is an analytic path in 
$\mathcal T(S)$, then we define
$$\|\dot\eta_0\|_\PP^2=||d\Phi(\dot\eta_0)||_\PP^2.$$
Proposition \ref{pressure form is a Hessian} will allow us to identify the pressure form with
the Hessian of the  intersection number $\II$.

\begin{theorem}{\rm (Thurston, Wolpert \cite{wolpert-thurston}, McMullen \cite{mcmullen-pressure})}
The pressure form is an analytic Riemannian metric
on $\mathcal{T}(S)$ which is invariant under the mapping class group and independent of the reference metric $\rho_0.$
Moreover, the resulting pressure metric is a constant multiple of the Weil-Petersson metric on
$\mathcal{T}(S)$.
\end{theorem}

\begin{proof} 
We first show that the pressure form is non-degenerate, so gives rise to a Riemannian metric.
Consider an analytic path $\{\eta_t\}_{(-1,1)}\subset \mathcal T(S)$. If $\| d\Phi(\dot\eta_0)\|_\PP=0$, then
Lemma \ref{degenerate vectors} implies that if $\gamma\in \pi_1(S)$, then
\begin{equation}\label{teich critical length} \left.\frac{\partial}{\partial t}\right|_{t=0} \ell_{\eta_t}(\gamma)=0.\end{equation}
However, there exist $6g-5$ elements $\{\gamma_1,\ldots,\gamma_{6g-5}\}$ of $\pi_1(S)$, so
that the mapping from $\mathcal{T}(S)$ into $\mathbb R^{6g-5}$ given by taking $\rho$ to
$(\ell_\rho(\gamma_i))_{i=1}^{6g-5}$ is a real analytic embedding (see Schmutz \cite{schmutz}).
Therefore,
since $\left.\frac{\partial}{\partial t}\right|_{t=0} \ell_{\eta_t}(\gamma_i)=0$ for all $i$, we conclude that $\dot\eta_0=0$.
Therefore, the pressure form is non-degenerate.

If $\rho,\eta\in\mathcal T(S),$ the intersection number 
$$\II(\rho,\eta)=\II(f_\rho,f_\eta)=\lim_{T\to\infty} \frac{1}{\#R_T(\rho)}\sum_{[\gamma]\in R_T(\rho)} \frac{\ell_{\sigma}(\gamma)}{\ell_{\rho}(\gamma)}$$ 
where $\ell_\rho(\gamma)$ is the translation length of $\rho(\gamma)$ and $R_T(\rho)$ is the collection of
conjugacy classes of elements of $\pi_1(S)$ whose images have translation length at most $T$.
So, $\II$ is independent of the reference metric $\rho_0$ 
and invariant by the action of  the mapping class group of $S$.
Proposition \ref{pressure form is a Hessian} states  that
$$\|\dot\eta_0\|_\PP=\left.\frac{\partial^2}{\partial t^2}\right|_{t=0}
\JJ(f_{\eta_0},f_{\eta_t}) =\left.\frac{\partial^2}{\partial t^2}\right|_{t=0}\II(f_{\eta_0},f_{\eta_t})$$
(again by fact (3)) and thus the pressure metric is mapping class group invariant.

One may interpret $\II(\rho,\eta)$ as the length in $X_\eta$ of a random unit length geodesic on $X_\rho$. So,
the pressure metric is given by considering the Hessian of the length of a random geodesic.
Since the pressure form agrees with Thurston's metric, Wolpert's work \cite{wolpert-thurston} implies that
the pressure metric is a multiple of the Weil-Petersson metric.
\end{proof}

\section{The pressure metric on the Hitchin component}

Let $V$ be a vector space and $G$ be a group. Recall that a representation
$\tau:G\to\ms{GL(V)}$ is \emph{irreducible} if $\tau(G)$ has no proper invariant subspaces other than $\{0\}.$ 
Let us begin by recalling the following well known result, see for example Humphreys \cite{humphreys}.

\begin{proposition} For each integer $d\geq2$ there exists an irreducible representation
$\tau_d: \ms{{SL}_2(\Real)}\to\ms{{SL}_d(\Real)},$ unique up to conjugation.
\end{proposition}

The existence of such an irreducible action is an explicit construction we will now explain. 
Denote by $\ms{Sym}^d(\Real^2)$ the $d$-dimensional vector space of homogeneous polynomials 
on 2 variables of degree $d-1.$ A base for $\ms{Sym}^d(\Real^2)$ is, for example, 
$$\mathcal B=\{x^{d-1},x^{d-2}y,\cdots,xy^{d-2},y^{d-1}\}.$$ 
We identify $x$ with $(1,0)$ and $y$ with $(0,1)$ in $\Real^2$ so that if 
$g=(\begin{smallmatrix} a & b \\ c & d\end{smallmatrix})\in\ms{SL}_2(\Real)$, 
then $g\cdot x=ax+cy$ and $g\cdot y=bx+dy.$ The action of $\ms{SL}_2(\Real)$ on $\ms{Sym}^d(\Real^2)$ is 
defined on the base $\mathcal B$ by 
$$\tau_d(g)\cdot x^ky^{d-k-1}=(g\cdot x)^k(g\cdot y)^{d-k-1}.$$

As before, let $S$ be a closed oriented surface of genus $g\geq2.$ 
Hitchin \cite{hitchin} studied the  components of the space 
$$\hom(\grf,\ms{PSL_d(\Real))}/\ms{PGL_d(\Real)}.$$
containing an element $\rho:\grf\to\ms{PSL}_d(\Real)$ 
that factors as $$\grf\stackrel{\rho_0}{\longrightarrow}\ms{PSL_2(\Real)}\stackrel{\tau_d}{\longrightarrow}\psln,$$ 
where $\rho_0\in\mathcal T(S).$ 

By analogy with Teichm\"uller space, he named these components {\em Teichm\"uller components}, but they are now known as
\emph{Hitchin components}, and denoted by $\mathcal H_d(S).$
Each Hitchin component contains a copy of $\mathcal T(S)$, known as   \emph{the Fuchsian locus},
which is an image of $\mathcal{T}(S)$
under the mapping induced by $\tau_d$. Hitchin  proved the following remarkable result.

\begin{theorem}[Hitchin \cite{hitchin}] Each Hitchin component $\mathcal H_d(S)$ is an analytic manifold diffeomorphic to
$\Real^{(d^2-1)(2g-2)}=\Real^{|\chi(S)|\dim\ms{PSL}_d(\Real)}.$
\end{theorem}

Hitchin \cite{hitchin} commented that 
{\em ``Unfortunately, the analytical point of view used for the proofs 
gives no indication of the geometrical significance of the Teichm\"uller component.''}
Labourie \cite{labourie-anosov} introduced dynamical techniques to show that Hitchin representations,
i.e. representations in the Hitchin component, are geometrically meaningful. In particular, Hitchin
representations are discrete, faithful, quasi-isometric embeddings. Labourie's work significantly expanded
the analogy between Hitchin components and Teichm\"uller spaces.

We view the following result as a further step in exploring this analogy. Its proof follows the same basic strategy as in the 
Teichm\"uller space setting, although there are several additional difficulties to overcome.

\begin{theorem}
\label{hitchin metric}
{\rm (Bridgeman-Canary-Labourie-Sambarino \cite{BCLS})}
There exists an analytic Riemannian metric on $\mathcal H_d(S)$ which is invariant under the action of the mapping class group 
and restricts to a multiple of the Weil-Petersson metric on the Fuchsian locus.
\end{theorem}

\noindent
{\bf Remark:}
\begin{itemize}\item The mapping class group, regarded as a subgroup of ${\rm Out}(\pi_1(S)),$ acts by 
precomposition on $\mathcal H_d(S)$.

\item When $d=3$ metrics have also been constructed by Darvishzadeh-Goldman \cite{DG} and Qiongling Li \cite{li}.
Li \cite{li} showed that both her metric and the metric constructed by Darvishzadeh and Goldman have the properties
obtained in our result.
\end{itemize}

\subsection{Labourie's work} 
Labourie developed the theory of Anosov representations as a tool to study Hitchin representations. This theory
was further developed by Guichard and Wienhard \cite{guichard-wienhard} and 
has played a central role in the subsequent development of higher Teichm\"uller theory.
The following theorem summarizes some of the
major consequences of Labourie's work for Hitchin representations.

\begin{theorem}
\label{hitchin anosov}
{\em (Labourie \cite{labourie-anosov,labourie-energy})}
If $\rho\in \mathcal H_d(S)$  then
\begin{enumerate}
\item
$\rho$ is discrete and faithful,
\item
If $\gamma\in \pi_1(S)$ is non-trivial, then $\rho(\gamma)$ is diagonalizable over $\mathbb R$ with distinct eigenvalues.
\item
$\rho$ is a quasi-isometric embedding. 
\item
$\rho$ is irreducible.
\end{enumerate}
\end{theorem}

Theorem \ref{hitchin anosov} is based on Labourie's proof that Hitchin representations are Anosov with respect to a minimal
parabolic subgroup for $\psln$, i.e. the upper triangular matrices in $\psln$. We will develop the terminology necessary to
give a definition.

A \emph{complete flag} of $\Real^d$ is a sequence of vector subspaces $\{V_i\}_{i=1}^d$ such that  $V_i\subset V_{i+1}$ and 
$\dim V_i=i$ for all $i=1,\ldots, d$. Two flags $\{V_i\}$ and $\{W_i\}$ are \emph{transverse} if  $V_i\cap W_{d-i}=\{0\}$ for all $i$. 
Denote by $\mathscr F$ the space of complete flags and by $\mathscr F^{(2)}$ the space of pairs of transverse flags. The following result
should be viewed as the analogue of the limit map constructed in Proposition \ref{morse}.

\begin{theorem}[Labourie \cite{labourie-anosov}]\label{equivariant} 
If $\rho\in\mathcal H_d(S)$, then there exists a unique $\rho$-equivariant H\"older  map $\xi_\rho:\bgrf\to\mathscr F$ such that, if $x\neq y$,
then the flags $\xi_\rho(x)$ and $\xi_\rho(y)$ are transverse.
\end{theorem}

If $\rho\in\mathcal H_d(S)$ and $x\in\bgrf$ then we denote by $\xi^{(k)}_\rho(x)$ the $k$-th dimensional space in the flag $\xi_\rho(x).$ Notice that if $\gamma_+$ is an attracting fixed point of the action of $\gamma\in\grf$ on $\bgrf$, 
then $\xi_\rho^{(k)}(\gamma_+)$ is spanned by the eigenlines of $\rho(\gamma)$ associated to the $k$ eigenvalues of largest modulus.
In particular, $\xi^{(1)}_\rho(\gamma_+)$ is the attracting fixed point for the action of $\rho(\gamma)$ on $\mathbb P(\Real^d)$ and 
$\xi_\rho^{(d-1)}(\g_-)$ is its repelling hyperplane (where $\gamma^-$ is the repelling fixed point for the action of $\g$ on $\grf$).

When $\rho$ is Fuchsian, Labourie's map is an explicit construction, called the \emph{Veronese embedding}, 
which is moreover $\tau_d$-equivariant. This is a map from $\partial_\infty\H^2=\P(\Real^2)\to\mathscr F$ explicitly defined, identifying $\Real^d$ with $\ms{Sym}^d(\Real^2),$ by $$\Real\cdot (ax+by)\mapsto \{p\in\ms{Sym}^d(\Real^2) : p\textrm{ has }(ax+by)^{d-k}\textrm{ as a factor}\}_{k=1}^d.$$

\medskip\noindent
{\bf Conventions:} As in the previous section, we fix $\rho_0\in\mathcal{T}(S)$, so that 
$\rho_0$ identifies $S$ with $X_{\rho_0}$, and hence identifies
$\bgrf$ with $\partial_\infty\mathbb H^2$ and $\TT^1S$ with $\TT^1X_{\rho_0}$. Let $\phi=\phi^{\rho_0}$ be the geodesic flow on $S$.
Let
$$\bgrf^{(2)}=\{(x,y)\in\bgrf^2:x\neq y\}$$ 
and consider the Hopf parametrization of $\TT^1\mathbb H^2$ by $\bgrf^{(2)}\times\mathbb R$
where $(x,y,t)$ is the point on the geodesic $L$ joining $x$ to $y$ which is a (signed) distance $t$ from
the horocycle through $y$ and a fixed basepoint for $\mathbb H^2$.

\medskip

Labourie considers the bundle $E_\rho$ over $\TT^1S$ which is the quotient of 
\hbox{$\TT^1\mathbb H^2\times \mathscr F$}
by $\pi_1(S)$ where $\gamma\in\pi_1(S)$ acts on $\TT^1\mathbb H^2$ by $\rho_0(\gamma)$ and acts on $\mathscr F$ by
$\rho(\gamma)$. 
There is a flow $\tilde\psi^\rho$ on $\TT^1\mathbb H^2\times \mathscr F$
which acts by the geodesic flow on $\TT^1\mathbb H^2$ and acts trivially on $\mathscr F$. 
The flow $\tilde\psi^\rho$ descends to a flow $\psi^\rho$ on $E_\rho$. 
The limit map $\xi_\rho:\bgrf\to\mathscr F$
determines a section $\tilde\sigma_\rho:\TT^1\mathbb H^2\to \tilde E_\rho$ given by $\tilde\sigma(x,y,t)=((x,y,t),\xi_\rho(x))$ which
descends to a section $\sigma:\TT^1S\to E_\rho$.

A representation $\rho:\pi_1(S)\to\psln$ is {\em Anosov with respect to a
minimal parabolic subgroup} if and only if there is a limit map with the properties in Theorem \ref{equivariant}
such that the inverse  of the associated flow
$\psi^\rho$ is contracting on $\sigma_\rho(\TT^1S)$.

\subsection{The geodesic flow of a Hitchin representation}
\label{geodesic flow}

We wish to associate a topologically transitive metric Anosov flow to each Hitchin representation.
Since $\rho$ is discrete and faithful, one is tempted to  consider the geodesic flow of the associated locally symmetric space 
$$N_\rho=\rho(\grf)\backslash\psln/\ms{PSO}(d).$$ However, $N_\rho$ is neither closed, nor negatively curved, so its geodesic flow will not be Anosov. 
Moreover, this flow does not even have a nice compact invariant set where it is metric Anosov (see Sambarino \cite[Prop.  3.5]{sambarino-orbital}).

Sambarino \cite[\S 5]{sambarino-quantitative} (or more specifically \cite[Thm 3.2, Cor. 5.3 and Prop. 5.4]{sambarino-quantitative}) constructed metric Anosov flows associated to Hitchin representations which are H\"older orbit equivalent to a geodesic flow on a hyperbolic surface such that the closed orbit associated to $\g\in\grf$ has period $\log\Lambda_\gamma(\rho),$ where $\Lambda_\gamma(\rho)$ is the spectral radius of $\rho(\gamma)$, i.e.  the largest modulus of the eigenvalues of $\rho(\gamma)$. 

We will
use these flows to construct a thermodynamic mapping and an associated pressure metric satisfying the conclusions of Theorem \ref{hitchin metric}.

\begin{proposition}[{Sambarino \cite[\S 5]{sambarino-quantitative}}]
\label{functionHitchin} 
For every $\rho\in\mathcal H_d(S)$, there exists a positive H\"older function $f_\rho:\TT^1S\to(0,\infty)$ such that 
$$\int_{[\g]}f_\rho=\log\Lambda_\g(\rho)$$
for every $\g\in\grf$.
\end{proposition}

Notice that $\phi^{f_\rho}$ is a topologically transitive, metric Anosov flow,
H\"older orbit equivalent to the geodesic flow, whose periods are the logarithms of the spectral radii of $\rho(\grf)$. We call this flow
the {\em geodesic flow of the Hitchin representation}.

We will give a different construction of the geodesic flow of a Hitchin representation, from \cite{BCLS}, which generalizes easily
to the setting of projective Anosov representations of a word-hyperbolic group into $\sln$.
If \hbox{$\rho\in\mathcal H_d(S)$}, we let $\xi_\rho:\bgrf\to\mathscr F$ be the associated limit map to the 
space of complete flags where $\xi_\rho(x) = \{\xi_\rho^k(x)\}_{k=1}^d$. 
We consider the line bundle $F_\rho$ over $\bgrf^{(2)}$ whose fiber at $(x,y)$ is
$\ms M(x,y)=$ $$\{(\varphi,v)\in(\Real^d)^*\times\Real^d\mid \ker\varphi=\xi_\rho^{(d-1)}(x),\,v\in\xi_\rho^{(1)}(y), \varphi(v)=1\}/(\varphi,v)\sim(-\varphi,-v).$$
Consider the flow $\tilde\phi^\rho=(\tilde\phi^\rho_t: F_\rho\to F_\rho)_{t\in\Real}$ given by 
$$\tilde \phi^\rho_t(\varphi,v)=(e^{-t}\varphi,e^t v).$$ 
Notice that the $\pi_1(S)$l on $\tilde F_\rho$ given by 
$$\g(x,y,(\varphi,v))=(\gamma(x),\gamma(y),(\varphi\circ\rho(\g)^{-1},\rho(\gamma)v))$$
is free.  We further show that it is properly discontinuous and co-compact, so $\tilde\phi^\rho$ descends to a flow $\phi^\rho$ on $\ms{U}_\rho=F_\rho/\pi_1(S),$
which we call the \emph{geodesic flow} of $\rho.$ The proof proceeds by finding a $\rho$-equivariant orbit equivalence between 
$\TT^1\H^2$ and $F_\rho$.

\begin{proposition}
\label{flow orbit equivalent}
{\rm (\cite[Prop. 4.1+Prop 4.2]{BCLS})}
The group $\pi_1(S)$ acts properly discontinuous and cocompactly on $F_\rho.$  The quotient flow $\phi^\rho$  on $\ms{U}_\rho$ 
is H\"older orbit equivalent to the geodesic flow on $\TT^1S$. 
Moreover, the closed orbit associated to $\gamma\in\pi_1(S)$
has $\phi^\rho$-period $\log\Lambda_\rho(\gamma).$
\end{proposition}

\noindent
{\em Sketch of proof:}
Consider the flat bundle $E_\rho$ over $\TT^1S$ which is the quotient of 
$\TT^1\mathbb H^2\times \mathbb R^d$
by $\pi_1(S)$ where $\gamma\in\pi_1(S)$ acts on $\TT^1\mathbb H^2$ by $\rho_0(\gamma)$ and acts on $\mathbb R^d$ by
$\rho(\gamma)$. 
One considers a flow $\tilde\psi^\rho$ on $\TT^1\mathbb H^2\times \mathbb R^d$
which acts as the geodesic flow on $\TT^1\mathbb H^2$ and acts trivially on $\mathbb R^d$. 
The flow $\tilde\psi^\rho$ preserves the $\rho(\pi_1(S))$-invariant line sub-bundle $\tilde\Sigma$ whose fiber over
the point $(x,y,t)$ is $\xi^{(1)}_\rho(x)$.
Thus, $\tilde\psi^\rho$ descends to a flow $\psi^\rho$ on $E_\rho$ preserving the line sub-bundle
$\Sigma$ which is the quotient of $\tilde\Sigma$. Since $\rho$ is Anosov with respect to a minimal parabolic subgroup, $\psi^\rho$ is
contracting on $\Sigma$ (see \cite[Lem. 2.4]{BCLS}).
Since $\psi^\rho$ is contracting on $\Sigma$ one may use an averaging procedure to
construct a metric $\tau$ on $\Sigma$ with respect to which $\psi^\rho$ is uniformly contracting.

\begin{lemma}
\label{contracting metric}
{\rm (\cite[Lemma 4.3]{BCLS})}
There exists a H\"older metric $\tau$ on $\Sigma$ and $\beta>0$, so that for all $t>0$,
$$(\psi^\rho_t)_*(\tau)<e^{-\beta t}\tau.$$
\end{lemma}

We construct a $\rho$-equivariant 
H\"older orbit equivalence
$$\tilde j(x,y,t)=(x,y,u(x,y,t))$$
where $\tilde\tau(u(x,y,t))=1$ and $\tilde \tau$ is the lift of $\tau$ to $\tilde\Sigma$.
The map $\tilde j$ is $\rho$-equivariant, since $\xi_\rho$ is, and the fact that
$\tilde \tau$ is uniformly contracting implies that $\tilde j$ is injective. It remains
to prove that $\tilde j$ is proper to show that it is a homeomorphism. (We refer
the reader to the proof of Proposition 4.2 in \cite{BCLS} for this relatively simple argument.)
Then, $\tilde j$ descends to a H\"older orbit equivalence $j$ between $\TT^1S$ and $U_\rho$.

In order to complete the proof, it  suffices to evaluate the period of the closed
orbit associated to an element $\gamma\in\pi_1(S)$. The closed orbit associated to $\gamma$
is the quotient of the fiber of $F_\rho$ over $(\gamma^+,\gamma^-)$. If we pick $v\in\xi_\rho^{(1)}(\gamma^+)$ and
$\varphi\in \xi_\rho^{(d-1)}(\gamma^-)$ so that $\varphi(v)=1$,
then 
\begin{eqnarray} 
\gamma (\gamma^+,\gamma^-,(\varphi,v)) &=& (\gamma^-,\gamma^+,(\pm(\Lambda_\rho(\gamma))^{-1}\varphi,\pm \Lambda_\rho(\gamma) v )) \nonumber \\
&=&  \tilde\phi^\rho_{\log\Lambda_\rho(\gamma)}(\gamma^+,\gamma^-,(\varphi,v)),
\end{eqnarray}
so the closed orbit has period $\log\Lambda(\rho(\gamma))$ as claimed.
\qed

\medskip

Notice that Proposition \ref{functionHitchin} follows immediately from Proposition \ref{flow orbit equivalent} and
Lemma \ref{time-conjugacy}.

\subsection{The thermodynamic mapping}

Proposition \ref{functionHitchin} allows us to construct a thermodynamic mapping in the Hitchin setting.
Liv\v sic's Theorem (Theorem \ref{livsic}) guarantees that 
if $\rho\in\mathcal H_d(S)$, then the Liv\v sic cohomology class of the reparametrization function $f_\rho$
is well-defined. So, applying Lemma \ref{entropia2}, we may define a thermodynamic mapping 
$$\Phi:\mathcal H_d(S)\to \mathcal H(\TT^1S)$$  
from the Hitchin component $\mathcal H_d(S)$ to the  space 
$\mathcal H(\TT^1S)$ of Liv\v sic cohomology classes of pressure zero H\"older functions on $\TT^1S$,
by letting
$$\Phi(\rho) =[-h(\rho)\ f_\rho].$$  
In order to construct an analytic pressure form, we need to know that $\Phi$ admits
local analytic lifts to the space $\mathcal P(\TT^1S)$ of pressure zero H\"older functions on $\TT^1S$.

\begin{proposition}[{\cite[Prop. 6.2]{BCLS}}]
\label{Labourie maps analytic} 
The mapping $\Phi$ admits local analytic lifts to the space $\mathcal P(\TT^1S)$, i.e.  each $\rho \in \mathcal H_d(S)$ has an open neighborhood $W$   and an analytic map $\tilde{\Phi}: W \rightarrow \mathcal P(\TT^1S)$ such that
$\Phi(\rho) = [\tilde{\Phi}(\rho)]$.
\end{proposition}

\noindent
{\em Sketch of proof:}
Let $\rho\in\mathcal H_d(S)$. Choose a neighborhood $V$ of $\rho$  which we may implicitly
identify with a submanifold of  $\ms{Hom}(\pi_1(S),\psln)$ (by an analytic map 
whose composition with the projection map
is the identity). 
Consider the $\mathscr F$-bundle $\tilde A= V\times  \TT^1\mathbb H^2\times \mathscr F$ over 
$V\times \TT^1\mathbb H^2$. There is a natural action of $\pi_1(S)$ on $\tilde A$ so that $\gamma\in\pi_1(S)$
take $(\eta,(x,y,t),F)$ to $(\eta,\gamma(x,y,t),\eta(\gamma)F))$ with quotient $A$. The limit map $\xi_\rho$ determines
a section $\sigma_\rho$ of $A$ over $\{\rho\}\times \TT^1S$. 

The geodesic flow on $\TT^1S$ lifts to a flow $\{\Psi_t\}_{t\in\mathbb R}$ on $A$ 
(whose lift to $\tilde A$ acts trivially in the $V$ and $\mathscr F$ direction).
The Anosov property of Hitchin representations implies that the inverse flow is contracting on 
$\sigma_\rho(\{\rho\}\times \TT^1S)$.
One may extend $\sigma_\rho$ to a  section 
$\sigma$ of $A$ over $V\times \TT^1S$ which varies analytically in the $V$ coordinate
(after first possibly restricting  to a smaller neighborhood of the lift of $\rho$). One may now
apply the machinery developed by Hirsch-Pugh-Shub \cite{hirsch-pugh-shub} (see also Shub \cite{shub}),
to find a  section $\tau$ of $A$ over $W\times \TT^1S$,
where $W$ is a sub-neighborhood
of $V$, so that the inverse flow preserves and is contracting along $\tau(W\times \TT^1S)$. Here the main
idea is to apply  the contraction mapping theorem cleverly to show that one may take
$$\tau(\eta,X)=\lim \Psi_{-nt_0}(\sigma(\eta,\Psi_{nt_0}(x)))$$
for some $t_0>0$ so that $\Psi_{-t_0}$ is uniformly contracting.  It follows from standard techniques that $\tau$ 
varies smoothly in the $W$ direction and that the restriction to $\{\eta\}\times \TT^1S$ is H\"older for all
$\eta\in W$. One must complexify the situation by considering representations
into $\ms{PSL_d(\mathbb C)}$ in order to verify that $\tau$ varies analytically in the $W$ direction.
(See Section 6 of \cite{BCLS} for more details). 

The section $\tau$ lifts to a section $\tilde\tau$ of $\tilde A$ which
is induced by a map 
$$\hat\xi:W\times \partial_\infty\pi_1(S)\to\mathscr F$$
which
varies analytically in the $W$ direction such that 
$$\hat\xi_\eta=\hat\xi(\eta,\cdot):\partial_\infty\pi_1(S)\to \P(\Real^d)$$
is $\eta$-equivariant and H\"older for all $\eta\in W$. The uniqueness of limit maps for Hitchin representations guarantees
that $\hat\xi_\eta=\xi_\eta$. So, $ \xi_\eta$ varies analytically over $W$.

One may then examine the proof of Proposition \ref{flow orbit equivalent} and apply an averaging
procedure, as in the Teichm\"uller space case, to produce an analytically varying family of H\"older
function $\{f_\eta\}_{\eta\in W}$, so that the reparametrization of the geodesic flow on $\TT^1S$ by $f_\eta$ has
the same periods as $U_\eta$. (Again to get analytic, rather than just smooth, variation one
must complexify the situation. See Section 6 of \cite{BCLS} for details.)
Therefore, the map 
$$\tilde\Phi:W\to \mathcal P(\TT^1S)$$ 
given by 
$$\tilde\Phi(\eta)= - h(\eta) f_\eta$$
is an analytic local lift of $\Phi$.
\qed

\subsection{Entropy and intersection numbers}
\label{hitchin entropy}

Proposition \ref{functionHitchin} allows us to define entropy and intersection numbers for
Hitchin representations.
If \hbox{$\rho\in \mathcal H_d(S)$}, let
$$R_T(\rho)=\{[\gamma]\in [\pi_1(S)]\ |\ \log(\Lambda_\rho(\gamma))\le T\}.$$
The {\em entropy} of $\rho$ is given by
$$h(\rho)=h(f_\rho)=\lim_{T\to\infty} \frac{\log \#R_T(\rho)}{T}  .$$
The {\em intersection number} of $\rho$ and $\eta$  in $\mathcal H_d(S)$ is given by
$$\II(\rho,\eta)=\II(f_\rho,f_\eta)=\lim_{T\to\infty}\ {1\over \#R_T(\rho)}\ 
\sum_{[\gamma]\in R_{\rho}(T)}{\log(\Lambda_{\eta}(\gamma))\over \log(\Lambda_{\rho}(\gamma))}$$
and their {\em renormalized intersection number} is
$$\JJ(\rho,\eta)=\JJ(f_\rho,f_\eta)={h(\eta)\over h(\rho)} \II(\rho,\eta).$$
Proposition \ref{Labourie maps analytic} and Corollary \ref{entropy analytic} immediately give:

\begin{corollary}
\label{Hitchin entropy analytic}
Entropy varies analytically over $\mathcal H_d(S)$ and intersection $\II$ and
renormalized intersection $\JJ$ vary
analytically over $\mathcal H_d(S)\times \mathcal H_d(S)$.
\end{corollary}

\medskip\noindent
{\bf Remark:} It follows from Bonahon's work \cite{bonahon}, that the intersection
number is symmetric on Teichm\"uller space.
However, it is clear that the intersection number is not symmetric on
the Hitchin component. For example, one may use the work of Zhang \cite{zhang1,zhang2} to exhibit for
all $K>1$ and $d\ge 3$, $\rho_1,\rho_2\in\mathcal H_d(S)$ such that 
\hbox{$\log \Lambda (\rho_1(\gamma))\ge K \log\Lambda(\rho_2(\gamma))$} for all $\gamma\in \pi_1(S)-\{id\}$,
so $\II(\rho_1,\rho_2)\ge K$ and $\II(\rho_2,\rho_1)\le {1\over K}$.
One expects that the renormalized intersection number is also asymmetric. 

\subsection{The pressure form}

We then define the  analytic {\em pressure form} on $\mathcal H_d(S)$ as the pullback of the {\em pressure form} 
on $\mathcal P(\TT^1S)$ 
using a lift of the thermodynamic mapping $\Phi.$ Explicitly,
if  $\{\eta_t\}_{t\in (-1,1)}$ is an analytic path in 
$\mathcal H_d(S)$ and $\tilde\Phi:U\to \mathcal P(\TT^1S)$ is an analytic lift of $\Phi$ defined on a
neighborhood $U$ of $\rho$, then we define
$$\|\dot\eta_0\|_\PP^2=||d\tilde\Phi(\dot\eta_0)||_\PP^2.$$

If \hbox{$\rho\in \mathcal H_d(S)$} and 
$\JJ_\rho:\mathcal H_d(S)\to \mathbb R$ is defined by 
$$\JJ_\rho(\eta)=\JJ(\rho,\eta)=\JJ(f_\rho,f_\eta),$$ 
then Proposition \ref{pressure form is a Hessian} implies that the pressure form  on $\TT_\rho \mathcal H_d(S)$
is the Hessian of  $\JJ_\rho$ at $\rho.$ Since the renormalized intersection number is mapping class group
invariant by definition, it follows that the pressure form is also mapping class group invariant.
Wolpert's theorem \cite{wolpert-thurston} implies that the restriction of the pressure form to the Fuchsian locus
is a multiple of the Weil-Petersson metric.
It only remains to show that the pressure form is positive definite, 
so gives rise to an analytic Riemannian metric on all of $\mathcal H_d(S)$. 

\subsection{Non-degeneracy of the pressure metric}

We complete the proof of Theorem \ref{hitchin metric} by proving:

\begin{proposition}
\label{nondegenerate}
The pressure form is non-degenerate at each point in $\mathcal H_d(S)$.
\end{proposition}

We note that each Hitchin component $\mathcal H_d(S)$ lifts to a component of
$$X(\pi_1(S),\sln)={\rm Hom}(\pi_1(S),\sln)/\sln$$
and we will work in this lift throughout the proof (see Hitchin \cite[Section 10]{hitchin}.)

In particular,  this allows us to  define, for all $\gamma\in\grf$,
an analytic function $\tr_\gamma:\mathcal H_d(S)\to \mathbb R$, where $\tr_\gamma(\rho)$
is the trace of $\rho(\gamma)$.

As in the Teichm\"uller case, the proof proceeds by applying Corollary \ref{degenerate vectors}. 
If $\{\eta_t\}_{(-\eps,\eps)}\subset\mathcal H_d(S)$ is a path such that $\|\dot\eta_0\|_\PP=\|d\Phi v\|_\PP=0$,
then  
\begin{equation}\label{degenerate path}\left.\frac{\partial}{\partial t}\right|_{t=0} h(f_{\eta_t}) \ell_{f_{\eta_t}}(\g)=0.\end{equation}
for all $\g\in\grf$.
The main difference is that entropy is not constant in the Hitchin component. 

If $\g\in\pi_1(S)$, we may think of $\log\Lambda_\gamma$ as an analytic function on 
$\mathcal H_d(S)$, where we recall that $\Lambda_\gamma(\rho)$ is the spectral radius
of $\rho(\gamma)$.
The following lemma is an immediate consequence of 
Equation (\ref{degenerate path})  (compare with Equation (\ref{teich critical length}) in Section \ref{teich-pressure}).

\begin{lemma}\label{hitchin length equation} 
If $v\in\TT_\rho \mathcal H_d(S)$ and $\|D_\rho\Phi(v)\|_\PP=0$,
then
$$D_{\rho}\log\Lambda_\g (v)= -\frac{D_\rho h(v)}{h(\rho)}\log\Lambda_\g(\rho)$$
for all $\g\in\grf$.
\end{lemma} 

Lemma \ref{hitchin length equation} implies that if  $v\in\TT_\rho \mathcal H_d(S)$ is a degenerate vector,
and we set $K=-\frac{D_\rho h(v)}{h(\rho)}$, then $D_\rho\log\Lambda_\g (v)=K\log\Lambda_\g(\rho)$
for all $\g\in\grf$. The next proposition, which is the key step in the proof of Proposition \ref{nondegenerate},
then 
guarantees that the derivative of the  trace function of every element is trivial in the direction $v$.

\begin{proposition}
\label{K=0 Hitchin} 
If $v\in\TT_\rho\mathcal H_d(S)$ and there exists $K\in\Real$ such that  
$$D_\rho\log\Lambda_\g (v)=K\log\Lambda_\g(\rho)$$ 
for all $\g\in\grf$, then $K=0$ and $D_\rho\tr_\g(v)=0 $ for all $\g\in\grf$.
\end{proposition}

The proof of Proposition \ref{nondegenerate}, and hence Theorem \ref{hitchin metric},
is then completed by applying the following standard lemma.

\begin{lemma}
\label{traces generate}
If $\rho\in \mathcal H_d(S)$, then $\{D_\rho\tr_\g\mid\g\in\pi_1(S)\}$ spans the cotangent space $\TT^*_\rho\mathcal H_d(S).$
\end{lemma}

Since every Hitchin representation is absolutely irreducible, Schur's Lemma can be used to show
that $\mathcal H_d(S)$ immerses in the $\ms{SL}_d(\mathbb C)$-character variety 
$X(\pi_1(S), {\rm SL}_d(\mathbb C))$. Lemma \ref{traces generate} then follows from standard facts 
about $X(\pi_1(S), \ms{SL}_d(\mathbb C))$ (see Lubotzky-Magid \cite{lubotzky-magid}).

\medskip\noindent
{\em Proof of Proposition \ref{K=0 Hitchin}:}
It will be useful to introduce some notation. If $M$ is a real analytic manifold, 
an analytic function $f:M \rightarrow \mathbb R$  has \hbox{{\em log-type} $K$} at $v\in\ms T_uM$ 
if $f(u)\ne 0$ and
$$
{\rm D}_u{\log(|f|)}(v)=K\log(|f(u)|).
$$

Suppose that $A \in \sln$ has real  eigenvalues $\{\lambda_i(A)\}_{i=1}^n$ where
$$|\lambda_1(A)|  > |\lambda_2(A)|  > \ldots > |\lambda_m(A)|.$$
If $\p_i(A)$ is the projection onto the $\lambda_i(A)$-eigenspace parallel to the hyperplane spanned by
the other eigenspaces, then
\begin{equation}
A = \sum_{k=1}^m \lambda_k(A)\p_i(A).
\label{expansion}
\end{equation}

We say that two infinite order elements of $\pi_1(S)$ are {\em coprime} if they do not share a common power.
The following lemma is an elementary computation (see Benoist \cite[Cor 1.6]{benoist2} or \cite[Prop. 9.4]{BCLS}).

\begin{lemma}
\label{typk}
If $\alpha $ and $\beta$ are coprime elements of $\pi_1(S)$ and $\rho\in\mathcal H_d(S)$, then
$$
\tr\Big(\p_1\big(\rho(\alpha)\big)\p_1\big(\rho(\beta)\big)\Big)=
\lim_{n\to\infty} \frac{\lambda_1\big(\rho(\alpha^n\beta^n)\big)}{\lambda_1\big(\rho(\alpha^n)\big)\lambda_1\big(\rho(\beta^n)\big)}\ne 0$$
and
$$
\tr\Big(\p_1\big(\rho(\alpha)\big)\rho(\beta)\Big)=\lim_{n\to\infty}\frac{\lambda_1\big(\rho(\alpha^n\beta)\big)}{\lambda_1\big(\rho(\alpha^n)\big)}\ne 0$$
for all $\rho\in\mathcal H_d(S)$.
\end{lemma}

The following rather technical lemma plays a key role in the proof of Proposition \ref{K=0 Hitchin}.

\begin{lemma}
\label{technical}
Suppose that $\{a_p\}_{p=1}^q, \{u_p\}_{p=1}^{q} ,\{b_s\}_{s=1}^\infty$, and $\{ v_s\}_{s=1}^\infty$ are collections of real numbers 
so that   $\{|u_p|\}_{p=1}^q$ and $\{|v_s|\}_{s=1}^\infty$  are strictly decreasing, each $u_p$ is non-zero,  
$$\sum_{p=1}^q n a_p u^n_p = \sum^\infty_{s=1} b_sv_s^n$$
for all $n>0$,
and $\sum^\infty_{s=1} b_sv_s$ is absolutely convergent. Then, $a_p = 0$ for all $p$.
\end{lemma}

\begin{proof} 
We may assume without loss of generality that each $b_s$ is non-zero.
We divide each side of the equality by $nu_1^n$, to see that 
$$a_1 =\lim_{n\to\infty}\sum^\infty_{s=1}\left(\frac{b_s}{n}\right)\frac{v_s^n}{u_1^n}$$
for all $n$. However, the right hand side of the equation can only be bounded as $n\to\infty$,
if $|v_1|\le |u_1|$. However, if $|v_1|\le |u_1|$, then the limit of the right hand side, as $n\to\infty$, must be 0
and we conclude that $a_1 = 0$.

We may iterate this procedure to conclude that $a_p = 0$ for all $p$.
\end{proof}

Suppose that  $\alpha, \beta \in \pi_1(S)$ are coprime.  We
consider the analytic function  \hbox{$F_n:\mathcal H_d(S)\to \mathbb R$} given by
$$F_n(\rho)= \frac{\tr\Big(\p_1(\rho(\alpha)) \rho(\beta^n)\Big)}{\lambda_1\big(\rho(\beta^n)\big)\tr\Big(\p_1\big(\rho(\alpha)\big)\p_1\big(\rho(\beta)\big)\Big)}.$$
Lemma \ref{typk} and the assumption of Proposition \ref{K=0 Hitchin}
imply that $F_n$ is of log-type $K$ at $v$ (see the proofs of Proposition 9.4 and Lemma 9.8 in \cite{BCLS}).
Using equation (\ref{expansion}) we have
$$\rho(\beta^n) = \sum_{k=1}^d \lambda_k\big(\rho(\beta)\big)^n\p_k\big(\rho(\beta)\big).$$
Thus, we can write $F_n$ as
$$F_n (\rho)= 1 + \sum_{k=2}^d \frac{\tr\Big(\p_1\big(\rho(\alpha)\big)\p_k\big(\rho(\beta)\big)\Big)}{\tr\Big(\p_1\big(\rho(\alpha)\big)\p_1\big(\rho(\beta)\big)\Big)}
\left(\frac{\lambda_k(\rho(\beta))}{\lambda_1(\rho(\beta))}\right)^n  = 1 + \sum_{k=2}^d f_k t_k^n$$
where 
$$f_k(\rho)=\frac{\tr\Big(\p_1\big(\rho(\alpha)\big)\p_k\big(\rho(\beta)\big)\Big)}{\tr\Big(\p_1\big(\rho(\alpha)\big)\p_1\big(\rho(\beta)\big)\Big)}\ne 0\qquad
\textrm{and}\qquad  t_k(\rho)=\left(\frac{\lambda_k(\rho(\beta))}{|\lambda_1(\rho(\beta))|}\right)\ne 0.$$
Since $F_n$ has log-type $K$ at $v$ and is  positive in some neighborhood of $\rho$, 
\begin{equation}
D_\rho F_n(v) =
\sum_{k=2}^d n f_k t_k^n \frac{\dot{t}_k}{t_k} + \sum_{k=2}^d \dot{f}_k t_k^n  =
K F_n(\rho)\log(F_n(\rho)),
\label{1}
\end{equation}
where $\dot{t}_k = D_\rho t_k(v)$ and $\dot{f}_k = D_\rho a_k(v)$.
In order to simplify the proof, we consider Equation (\ref{1}) for even powers.
Using the Taylor series expansion for $\log(1+x)$ and grouping terms
we have 
$$F_{2n}\log(F_{2n}) = \left(1 + \sum_{k=2}^d f_k t_k^{2n}\right)\log\left(1 + \sum_{k=2}^d f_k t_k^{2n}\right) = \sum_{s=1}^\infty c_s w_s^n$$
where $\{w_s\}$ is a strictly decreasing sequence of positive terms. 
We may again regroup terms to obtain
$$\sum_{k=2}^d  {2n}\left(\frac{f_k\dot{t}_k}{t_k}\right)t_k^{2n} =\sum_{s=1}^\infty c_sw_s^n-\sum_{k=2}^d\dot{f}_kt_k^{2n} =
\sum_{s=1}^\infty b_sv_s^n$$
where $\{v_s\}$ is a strictly decreasing sequence of positive terms.  So, letting $u_k = t_k^2$, we see that for all $n$
$$\sum_{k=2}^d  n\left(\frac{2 f_k\dot{t}_k}{t_k}\right)u_k^{n} = \sum_{s=1}^\infty b_sv_s^n.$$

Lemma \ref{technical} implies that $\frac{f_k\dot{t}_k}{t_k} = 0$ for all $k$, so $\dot{t}_k=0$ for all $k$.
Let $\lambda_{i,\beta}$ be the real-valued analytic function on $\mathcal H_d(S)$ 
given by $\lambda_{i,\beta}(\rho)=\lambda_i(\rho(\beta))$. Then,
$$ \frac{\dot{\lambda}_{k,\beta}\lambda_{1,\beta} - \dot{\lambda}_{1,\beta}\lambda_{k,\beta}}{\lambda_{1,\beta}^2}= 0.$$
So,
$$D_\rho(\log(|\lambda_{k,\beta}|))(v) = \frac{{\dot\lambda}_{k,\beta}}{\lambda_{k,\beta}}=
\frac{{\dot\lambda}_{1,\beta}}{\lambda_{1,\beta}} = 
D_\rho (\log(|\lambda_{1,\beta}|))(v) = K\log(|\lambda_{1,\beta}(\rho)|).$$
Since $\lambda_{d,\beta} = \frac{1}{\lambda_{1,\beta^{-1}}}$,
\begin{eqnarray*}
K\log(|\lambda_{1,\beta^{-1}}(\rho)|) &=& D_\rho  (\log(|\lambda_{1,\beta^{-1}}|))(v)  =
D_\rho  (\log(|\lambda_{d,\beta^{-1}}|))(v)\\
&=& -D_\rho( \log(|\lambda_{1,\beta}|))(v) = -K\log(|\lambda_{1,\beta}(\rho)|).
\end{eqnarray*}
Therefore, since $\log(|\lambda_{1,\beta^{-1}}(\rho)|)$ and $\log(|\lambda_{1,\beta}(\rho)|)$ are both positive,
$K = 0$, which implies that  $\dot\lambda_k(\beta) = 0$ for all $k$.

Since, $D_\rho\lambda_{i,\beta}(v)=0$ for all $i$ and all $\beta$, $D_\rho\tr_\beta=0$ for all 
$\beta\in\pi_1(S)$.
This completes the proof of Proposition \ref{K=0 Hitchin}, and hence Proposition \ref{nondegenerate} and
Theorem \ref{hitchin metric}.
\qed

\subsection{Rigidity theorems for Hitchin representations}

\subsubsection{Entropy rigidity}
Potrie and Sambarino recently showed that entropy is maximized only along the Fuchsian locus.
One may view this as an analogue of Bowen's celebrated result \cite{bowen-qf} that the topological
entropy of a quasifuchsian group is at least 1 and it is 1 if and only if the group is Fuchsian.

\begin{theorem}
\label{entropy maximal}
{\rm (Potrie-Sambarino \cite{potrie-sambarino})}
If $\rho\in \mathcal H_d(S)$, then $h(\rho)\le {2\over d-1}$. Moreover, 
if $h(\rho)={2\over d-1}$, then $\rho$ lies in the Fuchsian locus.
\end{theorem}

\noindent
{\bf Remarks:}
(1) Crampon \cite{crampon-rigidity} had earlier established that the entropy associated to Hilbert length (see
Section \ref{hilbert length})  of holonomies of strictly convex projective structures on closed hyperbolic
manifolds is maximal exactly at the representations into $\ms{SO}(d,1)$. In particular, the Hilbert length
entropy on $\mathcal H_3(S)$ is maximal exactly at the Fuchsian locus.

(2) Tengren Zhang \cite{zhang1,zhang2} showed that, for all $d$, there exist large families of
sequences of Hitchin representations with entropy converging to 0. Nie \cite{nie-simplicial}
had earlier constructed specific examples when $d=3$.

\subsubsection{Intersection number and marked length rigidity theorems}
One also obtains the following rigidity theorem for Hitchin representations with respect to the
intersection number. Notice that the definition of the renormalized intersection number $\JJ(\rho_1,\rho_2)$ 
for \hbox{$\rho_1\in \mathcal H_{d_1}(S)$} and \hbox{$\rho_2\in\mathcal H_{d_2}(S)$} makes sense
even if $d_1\ne d_2$, see Section \ref{hitchin entropy}. Moreover, if $f_1:\TT^1S\to\mathbb R$ and 
\hbox{$f_2:\TT^1S\to\mathbb R$} are positive H\"older
functions such that $\phi^{f_1}=U_{\rho_1}$ and $\phi^{f_2}=U_{\rho_2}$, then
$\JJ(\rho_1,\rho_2)=\JJ(f_1,f_2)$. In particular, see Lemma \ref{renormalized rigidity}, 
$\JJ(\rho_1,\rho_2)\ge 1$.

\begin{theorem}{\rm (\cite[Cor. 1.5]{BCLS})}
\label{hitch-rigid}
Let $S$ be a closed, orientable surface and let
$\rho_1\in \mathcal H_{d_1}(S)$ and $\rho_2\in \mathcal H_{d_2}(S)$ 
be two Hitchin representations such that  
$$
\JJ(\rho_1,\rho_2)=1.
$$
Then, either
\begin{enumerate}
\item $d_1=d_2$ and $\rho_1=\rho_2$ in $\mathcal H_{d_1}(S)$, or
\item there exists an element $\rho$ of the Teichm\"uller space $\mathcal T(S)$ so that
$\rho_1=\tau_{d_1}(\rho)$ and  $\rho_2=\tau_{d_2}(\rho)$.
\end{enumerate}
\end{theorem}

The proof of Theorem \ref{hitch-rigid} makes use of general rigidity results in the 
Thermodynamic Formalism and a result of Guichard \cite{guichard-zc}
classifying Zariski closures of images of Hitchin representations.

As an immediate corollary, one
obtains a marked length rigidity theorem where one uses the logarithm of the spectral radius as a notion of length.

\begin{corollary}
\label{rigid-cor}
If $\rho_1,\rho_2\in \mathcal H_d(S)$, then
$$\frac{h(\rho_1)}{h(\rho_2)}\sup_{\gamma\in\grf}\frac{\Lambda_\gamma(\rho_1)}{\Lambda_\gamma(\rho_2)}\geq1$$ 
with equality if and only  if 
there exists $g\in\ms{GL}_d(\mathbb R)$ such that $g\rho_1 g^{-1}= \rho_2.$ In particular, if $\Lambda_\gamma(\rho_1)=\Lambda_\gamma(\rho_2)$ for all $\gamma\in\grf,$ 
then $\rho_1$ and $\rho_2$ are conjugate in {$\ms{GL}_d(\mathbb R).$}
\end{corollary}

Bridgeman, Canary and Labourie have recently established that it suffices to consider lengths of
simple closed curves.

\begin{theorem}{\rm (Bridgeman-Canary-Labourie \cite{BCL})}
\label{simple length}
If $\rho_1,\rho_2\in \mathcal H_d(S)$ and  $\Lambda_\alpha(\rho_1)=\Lambda_\alpha(\rho_2)$ for all $\alpha\in\grf$
which are represented by simple closed curves, 
then $\rho_1$ and $\rho_2$ are conjugate in $\ms{PGL}_d(\mathbb R).$
\end{theorem}

\noindent
{\bf Remarks:} Burger \cite{burger} introduced a renormalized intersection number between 
convex cocompact representations into rank one Lie groups and proved an analogue of Theorem \ref{hitch-rigid} in that setting. One should compare Corollary \ref{rigid-cor}  and Theorem \ref{simple length}
with the marked length spectrum rigidity theorem of Dal'bo-Kim \cite{dalbo-kim} for Zariski dense representations. 
Both Dal'bo-Kim \cite{dalbo-kim} and Theorem \ref{hitch-rigid} rely crucially  on work of Benoist \cite[Thm. 1.2]{limite}.
However, the proof of Theorem \ref{simple length} uses Labourie's equivariant Frenet map into the flag variety, 
see Theorem \ref{equivariant}, and the theory of positive representations developed
by Fock and Goncharov \cite{fock-goncharov}.

\subsection{An alternate length function}
\label{hilbert length}
Throughout the section, we have used the logarithm of the spectral radius as a  notion of length.
It is also quite natural to consider the length of $\rho(\gamma)$ to be
$$\ell_{\ms H}(\rho(\gamma))=\log\Lambda(\rho(\gamma))+\log\Lambda(\rho(\gamma^{-1})).$$
For example, if $\rho\in \mathcal H_3(S)$, then $\rho$ is the holonomy of a convex projective structure on $S$,
and $\ell_{\ms H}(\rho(\gamma))$ is the translation length of $\gamma$ in the associated Hilbert metric on $S$.

Sambarino \cite{sambarino-quantitative} also proves that there is a reparametrization of $\TT^1S$ whose periods
are  given by $\ell_{\ms H}(\rho(\gamma))$.

\begin{proposition}
{\rm (Sambarino \cite[Thm. 3.2, \S5]{sambarino-quantitative})}
If $\rho\in\mathcal H_d(S)$, then there exists a positive H\"older function $f_\rho^{\ms H}:\TT^1S\to(0,\infty)$
such that 
$$\int_{[\g]}f_\rho^{\ms H}=\ell_{\ms H}(\rho(\gamma))$$
for all $\g\in\grf$.
\end{proposition}

We give a proof which uses a cross ratio to construct $f_\rho^{\ms H}$ from the limit map $\xi_\rho$, as is done in the Teichm\"uller setting.
It is adapted from the construction given in section 3 of Labourie \cite{labourie-energy}.

\begin{proof}
Given linear forms $\varphi, \psi\in(\Real^d)^*$ and vectors $v,w\in\Real^d$ such that $v\notin\ker\psi$ and $w\notin\ker\varphi,$ 
define the \emph{cross-ratio} 
$$[\varphi,\psi,v,w]=\frac{\varphi(v)\psi(w)}{\varphi(w)\psi(v)}.$$ 
Note that the cross ratio only depends on the projective classes of $\varphi$, $\psi$, $v$, and $w$,
and is invariant under $\psln.$ Moreover, if $g\in\psln$ is bi-proximal and $v\notin\ker g_-\cup\ker (g^{-1})_-$, then 
\begin{equation}\label{periods}[g_-,(g^{-1})_-,v,gv]=\Lambda(g)\Lambda(g^{-1})\end{equation}
where $g_-$ is a linear functional whose kernel is the repelling hyperplane of $g$.

Theorem \ref{equivariant} provides a $\rho$-equivariant map $\xi_\rho:\partial_\infty\H^2\to\mathscr F.$ 
Define \hbox{$\kappa_\eta:\TT^1\H^2\times\Real\to\Real$} by 
$$\kappa_\eta((x,y,z),t) = \log\left|[\xi_\eta^{(d-1)}(x),\xi_\eta^{(d-1)}(z), \xi_\eta^{(1)}(y),\xi_\eta^{(1)}(u_t(x,y,z))]\right|$$ 
where  $u_t$ is determined by $\phi_t(x,y,z) = (x, u_t(x,y,z), z).$ Work of Labourie \cite[\S 3]{labourie-energy} implies that $t\mapsto \kappa_\eta((x,y,z),t)$ is an increasing homeomorphism of $\Real,$ so averaging $\kappa_\eta$ and taking derivatives as before 
provides the desired function $f_\eta^\ms H:\TT^1S\to(0,\infty)$.
Equation (\ref{periods}) implies that $f_\eta^\ms{H}$ has the desired periods.
\end{proof}

We may again obtain a  thermodynamic mapping $\Phi^\ms{H}:\mathcal H_d(S)\to\mathcal H(\TT^1S)$  defined by
$$\eta\mapsto[-h(f^\ms{\ms H}_\eta)f^\ms{H}_\eta].$$
One can use the same arguments as above to show that $\Phi^\ms{H}$ has locally analytic lifts 
and one can pull-back the pressure form  via $\Phi^\ms H$ to obtain an analytic pressure semi-norm
$\|\cdot\|_{\ms{H}}$ on $\TT\mathcal H_d(S)$.
(Pollicott and Sharp \cite{pollicott-sharp} previously proved that the entropy associated to $\ell_{\ms H}$ varies analytically over
$\mathcal H_d(S)$.)
However, this pressure form is degenerate in ways which are
completely analogous to the degeneracy of the pressure metric on quasifuchsian space discovered by Bridgeman \cite{bridgeman-wp}.

Consider the \emph{contragredient involution} $\sigma:\psln\to\psln$ given by $g\mapsto (g^{-1})^{\tp},$ where $\tp$ denotes the transpose operator 
associated to the standard inner product of $\Real^d.$
This involution induces an involution on the Hitchin component $\hat\sigma:\mathcal{H}_d(S)\to\mathcal{H}_d(S)$, where
$\hat\sigma(\rho)(\gamma)=\sigma(\rho(\gamma))$ for all $\gamma\in\grf$.
If $\eta\in\mathcal H_d(S)$ is a representation whose image lies in (a group conjugate to) ${\mathsf{Sp}}(2n,\Real)$ (if $d=2n$ )
or ${\mathsf{SO}}(n,n+1,\Real)$ (if $d=2n+1$), then $\hat\sigma(\eta)=\eta.$

Consider the tangent vectors in $\TT\mathcal H_d(S)$ which are reversed by $D\hat\sigma$, i.e. let
$$\pb=\{v\in \TT \mathcal H_d(S):D\hat\sigma (v)=-v\}.$$
The vectors in $\pb$ are degenerate for the pressure metric $\|\cdot\|_{\ms H}$.

\begin{lemma} If $v\in \pb$, then $\| v\|_{\ms H}=0.$
\end{lemma}

\begin{proof} 
Consider a path $\{\eta_t\}_{(-1,1)}\subset\mathcal H_d(S)$  so that $\hat\sigma(\eta_t)=\eta_{-t}$ for all \hbox{$t\in (-1,1)$}.
Then, $\ell_{\ms H}(\eta_t(\gamma))=\ell_{\ms H}(\eta_{-t}(\gamma))$ and $h(f_{\eta_t}^{\ms H})=h(f_{\eta_{-t}}^{\ms H})$ for all \hbox{$t\in (-1,1)$} and $\gamma\in\grf$.
Therefore,
$$\left.\frac{\partial}{\partial t}\right|_{t=0} h(f^\ms H_{\eta_t}) \ell_{f^\ms H_{\eta_t}}(\g)=0$$
for all $\g\in\grf$. Corollary \ref{degenerate vectors} then implies that $\|v\|_{\ms H}=0.$
\end{proof}

\noindent
{\bf Remark:} With a little more effort one may use the techniques of \cite{BCLS} to show that 
these are the only degenerate vectors for  $\|\cdot\|_{\ms H}$ and that $\|\cdot\|_{\ms H}$ induces a path
metric on the Hitchin component.

\section{Generalizations and consequences}
\label{general}

In \cite{BCLS} we work in the more general setting of Anosov representations of word hyperbolic
groups into semi-simple Lie groups. In this section, we will survey these more general results
and discuss some of the additional difficulties which occur. The bulk of the work in \cite{BCLS}
is done in the setting of projective Anosov representations into $\sln$. We note
that Hitchin representations are examples of projective Anosov representations as are 
Benoist representations, i.e. holonomy representations of closed strictly convex (real) projective
manifolds (see Guichard-Wienhard \cite[Prop. 6.1]{guichard-wienhard}).

\subsection{Projective Anosov representations} 
We first show that the pressure form gives an analytic Riemannian metric on the space of
(conjugacy classes of)  projective Anosov, generic, regular\footnote{A representation $\rho:\Gamma\to {\rm SL}_d(\mathbb R)$ is {\em regular} if it is a smooth point
of the algebraic variety ${\rm Hom}(\Gamma,{\rm SL}_d(\mathbb R))$.}, irreducible representations.
In order to define projective Anosov representations, we begin by recalling basic facts about
the geodesic flow of a word hyperbolic group.
 
Gromov \cite{gromov} first established that a word hyperbolic group $\Gamma$ has an associated
geodesic flow $U_\Gamma$. Roughly, one considers the obvious flow on the space of all
geodesics in the Cayley graph of $\Gamma$, collapses all geodesics joining two points
in the Gromov boundary to a single geodesic, and considers the quotient by the action of $\Gamma$.
We make use of the version due to Mineyev \cite{mineyev}
(see also Champetier \cite{champetier}).
Mineyev defines a proper cocompact action of $\Gamma$ on 
$\widetilde{U_\Gamma}=\partial_\infty\Gamma^{(2)}\times \mathbb R$
and a metric on $\widetilde{U_\Gamma}$, well-defined only up to H\"older equivalence,
so that $\Gamma$ acts by isometries, every orbit  of $\mathbb R$ is
quasi-isometrically embedded, and the $\mathbb R$-action is by Lipschitz homeomorphisms. Moreover,
the $\mathbb R$-action descends to a flow on $U_\Gamma=\widetilde{U_\Gamma}/\Gamma$.
In the case that $\Gamma$ is the fundamental group of a negatively curved manifold $M$, one
may take $U_\Gamma$ to be the geodesic flow on $\TT^1M$.

A representation $\rho:\Gamma\to \sln$ has {\em transverse projective limit maps} if there
exist continuous, $\rho$-equivariant limit maps
$$\xi_\rho:\partial_\infty\Gamma\to \mathbb P(\mathbb R^d)$$
and
$$\theta_\rho:\partial_\infty\Gamma\to Gr_{d-1}(\mathbb R^d)=\mathbb P((\mathbb R^d)^*)$$
so that if $x$ and $y$ are distinct points in $\partial_\infty\Gamma$, then
$$\xi_\rho(x)\oplus\theta_\rho(y)=\mathbb R^d.$$ 

A representation $\rho$ with transverse projective limit maps determines a splitting
$\Xi\oplus\Theta$ of the flat bundle $E_\rho$ over $U_\Gamma$. Concretely,  if $\tilde E_\rho$ is
the lifted bundle over $\widetilde{U_\Gamma}$, then the lift $\tilde \Xi$ of $\Xi$ has fiber
$\xi_\rho(x)$ and the  lift $\tilde \Theta$  of $\Theta$ has fiber $\theta_\rho(y)$ over the point $(x,y,t)$.
The geodesic flow on $U_\Gamma$ lifts to a flow on $\widetilde{U_\Gamma}$ which extends, trivially
in the bundle factor, to a flow on $\tilde E_\rho$ which descends to a flow on $E_\rho$.
One says that $\rho$ is {\em projective Anosov} if the resulting flow on the associated bundle
${\rm Hom}(\Theta,\Xi)=\Xi\otimes\Theta^*$ is contracting.

Projective Anosov representations are discrete, well-displacing, quasi-iso\-me\-tric embeddings with finite kernel such
that the image of each infinite order element is bi-proximal, i.e. its eigenvalues of maximal and minimal modulus have
multiplicity one (see 
Labourie \cite{labourie-anosov,labourie-energy} and Guichard-Wienhard \cite[Thm. 5.3,5.9]{guichard-wienhard}).
However, projective Anosov representations need not be irreducible and the images
of elements need not be diagonalizable over $\mathbb R$. 
On the other hand,
Guichard and Wienhard \cite[Prop. 4.10]{guichard-wienhard} showed that any 
irreducible representation with transverse projective limits maps is projective Anosov.

\subsection{Deformation spaces}
The space of all projective Anosov representations of a fixed word hyperbolic group $\Gamma$
into $\sln$ is an open subset of
${\rm Hom}(\Gamma,\sln)/\sln$ (see Labourie \cite[Prop. 2.1]{labourie-anosov} and Guichard-Wienhard \cite[Thm. 5.13]{guichard-wienhard}).
However, a projective Anosov representation need not be a smooth point of ${\rm Hom}(\Gamma,\sln)/\sln$
(see Johnson-Millson \cite{johnson-millson}). Moreover, the set of projective Anosov representations
need not be an entire component of ${\rm Hom}(\Gamma,\sln)/\sln.$

In order to have the structure of a real analytic manifold, we consider
the space $\widetilde{\mathcal{C}}(\Gamma,d)$  of regular, projective Anosov, irreducible representations 
$\rho:\Gamma\to \sln$
and let
$$\mathcal C(\Gamma,d)=\widetilde{\mathcal C}(\Gamma,d)/\sln.$$

If $\ms G$ is a reductive subgroup of $\sln$, we can restrict the whole discussion to
representations with image in $\ms G$, i.e. let $\widetilde{\mathcal{C}}(\Gamma,\ms G)$ 
be the space of regular,  projective Anosov, irreducible representations 
$\rho:\Gamma\to \ms G$ and let
$$\mathcal C(\Gamma,\ms G)=\widetilde{\mathcal C}(\Gamma,\ms G)/\ms G.$$

We will later want to restrict to the space
$\mathcal C_g(\Gamma,\ms G)$ of  $\ms G$-{\em generic} representation in $\mathcal C(\Gamma,\ms G)$, i.e.
representations such that the centralizer of some element in the image is a maximal torus in $\ms G$.
In particular, in the case that $\ms G=\sln$, a representation is $\ms G$-{\em generic} if
some element in the image is diagonalizable over $\mathbb C$
with distinct eigenvalues.
The resulting spaces are real analytic manifolds.

\begin{proposition}
\label{deformation spaces analytic}
{\rm (\cite[Prop. 7.1]{BCLS})}
If $\Gamma$ is a word hyperbolic group and $\ms G$ is a reductive subgroup of $\sln$,
then $\mathcal C(\Gamma,d)$, $\mathcal C(\Gamma,\ms G)$, $\mathcal C_g(\Gamma,\ms G)$
and $\mathcal C_g(\Gamma,d)=\mathcal C_g(\Gamma,\sln)$ are all real
analytic manifolds.
\end{proposition}

\subsection{The geodesic flow, entropy and intersection number}
One new difficulty which arises, is that it is not known in general whether or not the
geodesic flow of a word hyperbolic group is metric Anosov, i.e. a Smale flow in the sense
of Pollicott \cite{pollicott-smale}. Notice that our construction in Section \ref{geodesic flow}
immediately generalizes to give,  for any projective Anosov representation $\rho$, a geodesic flow $U_\rho$ 
which is H\"older orbit equivalent to $U_\Gamma$ and whose periods are exactly spectral radii of infinite
order elements of $\Gamma$. In general, we must further show \cite[Prop. 5.1]{BCLS}
that $U_\rho$ is a topologically transitive metric Anosov flow.

\begin{proposition}
\label{geodesic flow is Anosov}
{\rm (\cite[Prop. 4.1, 5.1]{BCLS})}
If $\rho:\Gamma\to\sln$ is projective Anosov, then there exists a
topologically transitive, metric Anosov flow $U_\rho$ which is  H\"older orbit equivalent
to $U_\Gamma$ such that the orbit associated to $\gamma\in\Gamma$ has period
$\Lambda(\rho(\gamma))$.
\end{proposition}

Lemma \ref{time-conjugacy}  provides a H\"older function $f_\rho:\ms{U}_\Gamma\to(0,\infty)$, well-defined up to Liv\v sic cohomology, such that
$U_\rho$ is H\"older conjugate to the reparametrization of $U_\Gamma$ by $f_\rho$.
One may then use the Thermodynamic Formalism to define the
entropy of a projective Anosov representation and the intersection number and renormalized
intersection number of two projective Anosov representations. If $\rho$ is projective Anosov, we 
define
$$R_T(\rho)=\{[\gamma]\in [\pi_1(S)]\ |\ \log(\Lambda_\g(\rho))\le T\}$$
and
the {\em entropy} of $\rho$ is given by
$$h(\rho)=h(f_\rho)=\lim_{T\to\infty} \frac{\log \# R_T(\rho)}{T} .$$
The {\em intersection number} of  two projective Anosov representations $\rho$ and $\eta$ is given by
$$\II(\rho,\eta)=\II(f_\rho,f_\eta)=\lim_{T\to\infty}\ {1\over \# R_{T}(\rho)}\ 
\sum_{[\gamma]\in R_{\rho}(T)}{\log(\Lambda_{\g}(\eta))\over \log(\Lambda_{\g}(\rho))}$$
and their {\em renormalized intersection number} is
$$\JJ(\rho,\eta)={h(\eta)\over h(\rho)} \II(\rho,\eta).$$
One may use the technique of proof of Proposition \ref{Labourie maps analytic} to show that
all these quantities vary analytically.

\begin{theorem}
\label{analytic in projective case}
{\rm (\cite[Thm. 1.3]{BCLS})}
If $\Gamma$ is a word hyperbolic group and $\ms G$ is a reductive subgroup of $\psln$, then entropy varies analytically over 
$\mathcal C(\Gamma,\ms G)$ and intersection number and renormalized intersection
number vary analytically over $\mathcal C(\Gamma,\ms G)\times\mathcal C(\Gamma,\ms G)$.
\end{theorem}

\subsection{The pressure metric for projective Anosov representation spaces}

If $\ms G$ is a reductive subgroup of $\psln$, we
define a thermodynamic mapping
$$\Phi:\mathcal C(\Gamma,\ms G)\to\mathcal H(\ms{U}_\Gamma)$$ 
by $\rho\mapsto [-h(f_\rho)f_\rho].$ We can again show that $\Phi$ has locally analytic lifts, so we can pull back the pressure norm on
$\mathcal P(U_\Gamma)$ to obtain a pressure semi-norm $\|\cdot\|_\PP$ on $\mathcal C(\Gamma,\ms G)$.
The resulting pressure semi-norm gives an analytic Riemannian metric on $\mathcal C_g(\Gamma,\ms G)$.

\begin{theorem}
{\rm (\cite[Thm. 1.4]{BCLS})}
\label{projective pressure metric} 
If $\Gamma$ is a word hyperbolic group and $\ms G$ is a reductive subgroup of $\sln$,
then the pressure form is an analytic
${\rm Out}(\Gamma)$-invariant Riemannian metric on $\mathcal C_g(\Gamma,\ms G)$.
In particular, the pressure form is an analytic
${\rm Out}(\Gamma)$-invariant Riemannian metric on $\mathcal C_g(\Gamma,d)$.
\end{theorem}

It only remains to prove that the pressure semi-norm is non-degenerate. We follow the same outline as in the Hitchin setting,
but encounter significant new technical difficulties.
As before, we may use Corollary \ref{degenerate vectors} to obtain restrictions on the derivatives of spectral length of group elements.

\begin{lemma} {\rm (\cite[Lem. 9.3]{BCLS})}
\label{entropy=K} 
If $\ms G$ is a reductive subgroup of $\psln$, $v\in \TT_\rho\mathcal C(\Gamma,\ms G)$ and $\| v\|_\PP=0,$ then
$$D_{\rho}\log\Lambda_{\g} (v)= -\frac{D_\rho h(v)}{h(\rho)}\log\Lambda_{\g}(\rho)$$
for all $\gamma\in\Gamma$.
\end{lemma}

We use this to establish the following analogue of Proposition \ref{K=0 Hitchin} from the Hitchin setting.
In order to do so, we must work in the setting of $\ms G$-generic representations and we can only conclude that the
derivative of spectral length, rather than trace, is trivial.

\begin{proposition}
\label{degenerate implies K=0}{\rm (\cite[Prop 9.1]{BCLS})} 
If $\ms G$ is a reductive subgroup of $\psln$, $v\in \TT_\rho\mathcal C_g(\Gamma,\ms G)$ and  there exists $K$ such that
$$D_{\rho}\log\Lambda_{\g} (v)=K\log\Lambda_{\g}(\rho)$$
for all $\gamma\in\Gamma$, then $K=0$. In particular,
$D_{\rho}\log \Lambda_\g (v)=0$
for all $\gamma\in\Gamma$.
\end{proposition}

One completes the proof by showing that the derivatives of the spectral radii functions generate the cotangent space.

\begin{proposition}{\rm (\cite[Prop. 10.3]{BCLS})}
If  $\ms G$ is a reductive subgroup of $\psln$ and $\rho\in\mathcal C_g(\Gamma,\ms G)$,
then the set $\{D_{\rho}\Lambda_\g\ \mid\ \g\in\Gamma\}$ spans $\TT_\rho^*\mathcal C_g(\Gamma,\ms G).$
\end{proposition}

\subsection{Anosov representations}

We now discuss the generalizations of our work to spaces of more general Anosov representations.
If $\ms G$ is any semisimple Lie group with finite center 
and $\ms P^\pm$ is a pair of  opposite parabolic subgroups, then
one may consider $(\ms G,\ms P^\pm)$-Anosov representations of a word
hyperbolic group $\Gamma$ into $\ms G$.
A $(\ms G,\ms P^\pm)$-Anosov representation $\rho:\Gamma\to\ms G$
has limit maps 
$$\xi_\rho^\pm:\partial_\infty\Gamma\to \ms G/\ms P^\pm$$
(which are transverse in an appropriate sense and give rise to associated flows with contracting/dilating properties).
In fact, Zariski dense representations with transverse limit maps are always $(\ms G,\ms P^\pm)$-Anosov
(\cite[Thm 4.11]{guichard-wienhard}). 

Projective Anosov representations are $(\ms G,\ms P^\pm)$-Anosov where $\ms G=\sln$, $\ms P^+$ is
the stabilizer of a line and $\ms P^-$ is the stabilizer of a complementary hyperplane
(\cite[Prop. 2.11]{BCLS}).
Hitchin  representations are $(\ms G,\ms P^\pm)$-Anosov where $\ms G=\sln$, $\ms P^+$ is the group
of upper triangular matrices (i.e. the stabilizer of the standard flag) and $\ms P^-$ is the group
of lower triangular matrices (Labourie \cite{labourie-anosov}).

We may think of Anosov representations as natural generalizations of Fuchsian 
representations, since
they are  discrete, faithful, quasi-isometric embeddings with finite kernel so that 
the image of every infinite order element is $\ms P^+$-proximal 
(\cite{labourie-anosov,labourie-energy} and \cite[Thm. 5.3,5.9]{guichard-wienhard}).
More generally, they may be thought of as generalizations of convex cocompact representations
into rank one Lie groups.
See Labourie \cite{labourie-anosov} and  Guichard-Wienhard \cite{guichard-wienhard} 
for definitions and more detailed discussions of Anosov representations. Gueritaud-Guichard-Kassel-Wienhard \cite{GGKW}
and Kapovich-Leeb-Porti \cite{KLP} have developed intriguing new viewpoints on Anosov representations
and their definition.

Guichard and Wienhard \cite[Prop. 4.2, Remark 4.12]{guichard-wienhard} (see also 
\cite[Thm 2.12]{BCLS}) showed that there exists
an irreducible representation $\sigma:\ms G\to \mathsf{SL}(V)$  (called the {\em Pl\"ucker representation})
such that  $\rho:\Gamma\to \ms G$ is
$(\ms G,\ms P^\pm)$-Anosov if and only if $\sigma\circ \rho$ is projective Anosov. Thus, one can
often reduce the study of  $(\ms G,\ms P^\pm)$-Anosov representations to the study of projective Anosov representations.

Let $\mathcal Z(\Gamma, \ms G, \ms P^\pm)$ be the space of (conjugacy classes of) regular, 
virtually Zariski dense $(\ms G,\ms P^\pm)$-Anosov representations. 
The space $\mathcal Z(\Gamma, \ms G, \ms P^\pm)$ is an analytic orbifold, which is a manifold if $\ms G$ is
connected (see \cite[Prop. 7.3]{BCLS}). The Pl\"ucker representation $\sigma:\ms G\to\sln$ allows one to view 
$\mathcal Z(\Gamma, \ms G, \ms P^\pm)$ as an analytically varying family of $\sigma(\ms G)$-generic
projective Anosov representations. One may pull back the pressure form and adapt the techiques from
the projective Anosov setting to prove:

\begin{theorem}
\label{general pressure metric}
{\rm (\cite[Cor. 1.9]{BCLS})}
If $\ms G$ is semi-simple Lie group with finite center and $\Gamma$ is word hyperbolic, then
the pressure form is an ${\rm Out}(\Gamma)$-invariant analytic Riemannian metric on 
$\mathcal Z(\Gamma, \ms G, \ms P^\pm)$.
\end{theorem}

\subsection{Examples}

There are two other important classes of higher Teichm\"uller spaces which are (quotients of)
entire components of representation varieties.

Burger, Iozzi and Wienhard \cite{BIW}  have studied representations of $\pi_1(S)$ into
a Hermitian  Lie group  $\ms G$ of tube type with maximal Toledo invariant, i.e.  maximal 
representations. Each maximal representation is Anosov, with respect to stabilizers of  points
in the Shilov boundary of the associated symmetric space (\cite{BILW,BIW2}), and the space of all maximal representations
is a collection of components of $\ms{Hom}(\pi_1(S),\ms G)$ (\cite{BIW}). One particularly nice case
arises when $\ms G=\ms{Sp}(4,\mathbb R)$, in which case there are $2g-3$ components which
are non-simply connected manifolds consisting entirely of Zariski dense representations
(see Bradlow-Garcia-Prada-Gothen \cite{BGG}). Hence,  the quotients by $\ms G$ of 
all such components admit pressure metrics.

Benoist \cite{benoist-divisible1,benoist-divisible3} studied holonomies of strictly convex projective 
structures  on a closed manifold $M$ and showed that these consist of  entire components of 
$\ms{Hom}(\pi_1(M),\psln)$. One may use his work to show that these representations, which we
call Benoist representations, are projective Anosov
(see Guichard-Wienhard \cite[Prop. 6.1]{guichard-wienhard}). Johnson-Millson \cite{johnson-millson} gave examples
of holonomy maps $\rho:\pi_1(M)\to \ms{SO}(d-1,1)$ of closed hyperbolic $d-1$-manifold, where $d\ge 5$,
such that $\rho$ is a singular point of $\ms{Hom}(\pi_1(M),\psln)$.

\subsection{Rank one Lie groups}
\label{rank one}
Let  $\Gamma$ be a word hyperbolic group and let $\mathsf G$ be a rank 1 semi-simple Lie group, 
e.g. ${\rm PSL}_2(\C)$. A representation
$\rho:\Gamma \to \mathsf G$ is {\em convex cocompact} if and only if whenever one chooses a basepoint
$x_0$ for the symmetric space $X=\ms K\backslash\ms G$ then the orbit map $\tau:\Gamma\to X$ given
by $\gamma\to \gamma(x_0)$ is a quasi-isometric embedding.
The limit set of $\rho(\Gamma)$ is
then the set of accumulation points in $\partial_\infty X$ of the image of the orbit map and one can
define the Hausdorff dimension of this set.
Patterson \cite{patterson}, Sullivan \cite{sullivan}, Corlette-Iozzi \cite{corlette-iozzi},
and Yue \cite{yue} showed that
the topological entropy of a convex cocompact representation agrees with the Hausdorff
dimension of the limit set of its image.

A representation $\rho:\Gamma\to \ms G$ is convex cocompact if and only if it is Anosov
(see Guichard-Wienhard \cite[Thm. 5.15]{guichard-wienhard}).
Since the Pl\"ucker embedding multiplies entropy by a constant 
depending only on $\ms G$ (see \cite[Cor. 2.14]{BCLS}),
the analyticity of the Hausdorff dimension of the limit set follows from the analyticity of entropy
for projective Anosov representations.

\begin{theorem}
{\rm (\cite[Cor. 1.8]{BCLS})}
If  $\Gamma$ is a word hyperbolic group and $\ms G$ is a rank 1 semi-simple Lie group,
then the Hausdorff dimension of the limit set  varies analytically over analytic families of 
convex cocompact representations of $\Gamma$ into $\mathsf G$.
\end{theorem}

\medskip\noindent
{\bf Remark:} When $\ms G=\ms{PSL}_2(\C)$, 
Ruelle \cite{ruelle-hd} proved this for surface groups and Anderson-Rocha 
\cite{anderson-rocha} proved it for  free products of surface groups and free groups.
Tapie \cite{tapie} used work of Katok-Knieper-Pollicott-Weiss \cite{KKPW} to show that the 
Hausdorff dimension is $C^1$ on smooth families of convex cocompact representations.

\medskip

Let $CC(\Gamma,\mathsf{PSL}_2(\C))$ be the space of (conjugacy classes of ) convex cocompact
representation of $\Gamma$ into $\ms{PSL}_2(\C)$. Bers \cite{bers} showed that
$CC(\Gamma,\mathsf{PSL}_2(\C))$ is an analytic manifold. Recall that 
a convex cocompact representation is not Zariski dense if and only
if it is virtually Fuchsian, i.e. contain a finite index subgroup conjugate into $\ms{PSL}_2(\mathbb R)$.
We may again use the Pl\"ucker representation to prove:

\begin{theorem} 
{\rm (\cite[Cor. 1.7]{BCLS})} If $\Gamma$ is word hyperbolic, then the pressure form
is ${\rm Out}(\Gamma)$-invariant  and analytic on
$CC(\Gamma,\mathsf{PSL}_2(\C))$ and is non-degenerate at any representation
which is not virtually Fuchsian. In particular, if 
$\Gamma$ is not either virtually free or virtually a surface group, then
the pressure form is an analytic Riemannian metric on \hbox{$CC(\Gamma,\mathsf{PSL}_2(\C))$}.
Moreover, the pressure form always induces a path metric on \hbox{$CC(\Gamma,\mathsf{PSL}_2(\C))$}.
\end{theorem}

Bridgeman \cite{bridgeman-wp} had previously defined and studied the pressure metric on
quasifuchsian space $QF(S)=CC(\pi_1(S),\ms{PSL}_2(\C))$. He showed that the degenerate vectors
in this case correspond exactly to pure bending vectors on the Fuchsian locus.

\subsection{Margulis space times} 

A {\em Margulis space time} is a quotient of $\mathbb R^3$ by a free, non-abelian group of affine
transformations which acts properly discontinuously on $\mathbb R^3$. They were originally discovered
by Margulis \cite{margulis-spacetime} as counterexamples to a question of Milnor. Ghosh \cite{ghosh-anosov}
used work of Goldman, Labourie and Margulis \cite{goldman-labourie,goldman-labourie-margulis} to interpret holonomy maps of Margulis
space times (without cusps) as ``Anosov representations'' into the (non-semi-simple) Lie group $\ms{Aff}(\mathbb R^3)$ of affine automorphisms
of $\mathbb R^3$.
Ghosh \cite{ghosh-pressure} was then able to adapt the techniques of \cite{BCLS} to produce a pressure form
on the analytic manifold $\mathcal M$ of (conjugacy classes of) holonomy maps of Margulis space times of fixed rank (with no cusps). This pressure form is
an analytic Riemannian metric on the slice $\mathcal M_k$ of $\mathcal M$ consisting of holonomy maps
with entropy $k$ (see Ghosh \cite[Thm. 1.0.1]{ghosh-pressure}), but has a degenerate direction on
$\mathcal M$, so  the pressure form has signature \hbox{$(\dim\mathcal M-1, 0)$} on $\mathcal M$.

\section{Open problems}

The geometry of the pressure metric is still rather mysterious and much remains to be explored.
The hope is that the geometry of the pressure metric will yield insights into the nature of the Hitchin
component and other higher Teichm\"uller spaces, in much the way that the study of the
Teichm\"uller and Weil-Petersson metrics  have been an important tool in our understanding of Teichm\"uller space 
and the mapping class group. It is natural
to begin by exploring analogies with the Weil-Petersson metric on Teichm\"uller space. 
We begin the discussion by recalling some basic properties of the Weil-Petersson metric.

\medskip\noindent
{\bf Properties of the Weil-Petersson metric:}
\begin{enumerate}
\item
The extended mapping class group 
is the  isometry group of $\mathcal T(S)$ in the Weil-Petersson metric (Masur-Wolf \cite{masur-wolf}).
\item
The Weil-Petersson metric is  negatively curved, but the sectional curvature is not bounded away from either 0 or $-\infty$ 
(Wolpert \cite{wolpert-neg}, Tromba \cite{tromba-neg}, Huang \cite{huang}).
\item
If $\phi$ is a pseudo-Anosov mapping class, then there is a lower bound for
its translation distance on Teichm\"uller space and there is a unique invariant geodesic axis for $\phi$
(Daskalopoulos-Wentworth \cite{DW-classification}).
\item
The Weil-Petersson metric is incomplete (Wolpert \cite{wolpert-incomplete}, Chu \cite{chu}).
However, it admits a metric
completion which is $CAT(0)$ and homeomorphic to the augmented Teichm\"uller space
(see Masur \cite{masur-wp} and Wolpert \cite{wolpert-wpcompletion}).
\end{enumerate}

Masur and Wolf's result \cite{masur-wolf} on the isometry group of $\mathcal T(S)$ suggests the following problem.

\medskip\noindent
{\bf Problem 1:} {\em Is the isometry group of a Hitchin component generated by
the (extended) mapping class group and the contragredient involution? More generally,
explore whether the relevant outer automorphism group is a finite index subgroup of
the isometry group of a higher Teichm\"uller space with the pressure metric.}

\medskip

Bridgeman, Canary, and Labourie  \cite{BCL} have shown that any diffeomorphism of $\mathcal H_3(S)$
which preserves the intersection number is an element of the extended mapping class group or the composition
of an element in the extended mapping class group with the contragredient involution. Along the way,
they show that any diffeomorphism which preserves the intersection number
also preserves the entropy and hence preserves  the renormalized intersection number, 
the pressure metric, and, by work of 
Potrie and Sambarino \cite{potrie-sambarino}, the Fuchsian locus. This suggests the following problem:

\medskip\noindent
{\bf Problem 2:} {\em Prove that if $g:\mathcal H_d(S)\to\mathcal H_d(S)$ is an isometry with respect to
the pressure metric then $\II(g(\rho),g(\sigma))=\II(\rho,\sigma)$
for all $\rho,\sigma\in\mathcal H_d(S)$. It would follow that the isometry of the group of the Hitchin component
$\mathcal H_3(S)$ is generated by the extended mapping class group of $S$ and the contragredient involution.}

\medskip

Bridgeman and Canary \cite{BCQF} have shown that the group of diffeomorphisms of quasifuchsian space $QF(S)$
which preserve the renormalized intersection number is generated by the extended mapping class group
and complex conjugation. So one may also consider the corresponding analogue of Problem 2 in quasifuchsian
space.

\medskip

It would be be useful to study the curvature of the pressure metric, guided by the results of
Wolpert \cite{wolpert-neg}, Tromba \cite{tromba-neg}, and Huang \cite{huang}.
Wolf's work \cite{wolf-hessian} (see also \cite{wolf-harmonic})
on the Hessians of length functions on Teichm\"uller space may offer a plan of attack here.

\medskip\noindent
{\bf Problem 3:} {\em Investigate the curvature of the Hitchin component in the pressure metric.}

\medskip

Pollicott and Sharp \cite{pollicott-sharp-graphs} have investigated the curvature of the pressure metric
on deformation spaces of marked metric graphs with entropy 1. In this setting, the curvature can be both
positive and negative. 

Labourie and Wentworth \cite{labourie-wentworth} have derived a formula for the pressure
metric at points in the Fuchsian locus of a Hitchin component in terms of Hitchin's parameterization
of the Hitchin component by holomorphic differentials. They also obtain variational formulas
which are analogues of classical results in the Teichm\"uller setting.

\medskip

Since Labourie \cite{labourie-energy} proved that the mapping class group acts properly
discontinuously on a Hitchin component, it is natural to study the geometry of this action.
One specific question to start with would be:

\medskip\noindent
{\bf Problem 4:} {\em Is there a lower bound for the translation distance for the action of a pseudo-Anosov
mapping class on the Hitchin component?}

\medskip

Since the restriction of the  pressure metric to the Fuchsian locus is a multiple of the Weil-Petersson metric,
the Hitchin component is incomplete and the metric completion contains augmented Teichm\"uller space.
However, very little is known about the completeness of the pressure metric in ``other directions.''
The work of Zhang \cite{zhang1,zhang2} and Loftin \cite{loftin-compactification} (when $d = 3$) should be relevant here. 
It may also be interesting to study the relationship between the metric completion and Parreau's compactification \cite{parreau}
of the Hitchin component.

\medskip\noindent
{\bf Problem 5:} {\em Investigate the metric completion of the Hitchin component or other
higher Teichm\"uller spaces.}

\medskip

Xu \cite{Xu} studied the pressure metric on the Teichmuller space ${\mathcal T}(S)$ where $S$ is a surface with non-empty geodesic boundary. He shows that the pressure metric  in this case is not equal to the classical Weil-Petersson metric on ${\mathcal T}(S)$. He further shows that it is not complete and  that the space of marked metric graphs on a fixed graph with its pressure metric arises naturally in the completion.  

\medskip

The following problem indicates  how little is known about the coarse geometry of the pressure metric.
We recall that a subset $A$ of a
metric space $X$ is said to be {\em coarsely dense} if there exists $D>0$ such that every point
in $X$ lies within $D$ of a point in $A$.

\medskip\noindent
{\bf Problem 6:} {\em (a) Is the Fuchsian locus coarsely dense in a Hitchin component?

(b) Is the Fuchsian locus coarsely dense in quasifuchsian space?

(c) If $M$ is an acylindrical 3-manifold with no toroidal boundary components and $\Gamma=\pi_1(M)$,
does $CC(\Gamma,\ms{PSL}_2(\C))$ have finite diameter?
}

\medskip

Zhang \cite{zhang1,zhang2} and Nie \cite{nie-simplicial} (when $d=3$) produce sequences in Hitchin components
where entropy converges to 0. These sequences are candidates to produce points arbitrarily far
from the Fuchsian locus. 

In case (c), ${\rm Out}(\Gamma)$ is finite (see Johannson \cite{johannson}) and 
$CC(\Gamma,\ms{PSL}_2(\C))$ has compact
closure in the $\ms{PSL}_2(\C)$-character variety (see Thurston \cite{thurston-acylindrical}).

\medskip

One may phrase all the above questions as being about the quotient of a higher Teichm\"uller space by its
natural automorphism group. Similarly, one might ask whether the quotient of the Hitchin component by
the mapping class group has finite volume.

\medskip\noindent
{\bf Problem 7:} {\em Does the quotient of the Hitchin component by the action of the mapping class group
have finite volume in the quotient pressure metric?}

\medskip

Potrie-Sambarino \cite{potrie-sambarino} showed that the entropy function is maximal
uniquely on the Fuchsian locus of a Hitchin component, so it is natural to investigate more
subtle behavior of the entropy function.

\medskip\noindent
{\bf Problem 8:}  {\em  Investigate the critical points on the entropy function.}

\medskip

Bowen \cite{bowen-qf} showed that the entropy function is uniquely minimal on the Fuchsian locus
in quasifuchsian space $QF(S)$. Bridgeman \cite{bridgeman-wp} showed that the entropy function on $QF(S)$ has no local maxima and
moreover the Hessian of the entropy function is positive-definite on at least a half-dimensional subspace at any critical point.

If $M$ is an acylindrical 3-manifold with no toroidal boundary components and $\Gamma=\pi_1(M)$,
then there is a unique representation in $CC(\Gamma,\ms{PSL}_2(\C))$ where the boundary of the
limit set of the image consists of round circles (see Thurston \cite{thurston-bangor}).
It is conjectured that the entropy has a unique minimum
at this representation (see Canary-Minsky-Taylor \cite{CMT}). Storm \cite{storm-volumemax} proved that this 
is the unique representation where
the volume of the convex core achieves its minimum.

\medskip

In the case of  $CC(\Gamma,\ms{PSL}_2(\C))$ we were able 
to obtain a path metric, even when the pressure form is degenerate on a submanifold. One might
hope to be able to do so in more general settings.

\medskip\noindent
{\bf Problem 9:} {\em If $\Gamma$ is a word hyperbolic group, $\ms G$ is a semisimple Lie group
and $\ms P^\pm$ is a pair of opposite parabolic subgroups, can one extend the pressure
metric on $\mathcal Z(\Gamma, \ms G, \ms P^\pm)$ to a path metric on the space of
all (conjugacy classes of ) $(\ms G,\ms P^\pm)$-Anosov representations of $\Gamma$ into $\ms G$?}

\end{document}